\documentclass[12pt]{article}

\usepackage[T1]{fontenc}
\usepackage[utf8]{inputenc}
\usepackage{amsmath,amssymb,amsthm,mathtools}
\usepackage{enumitem}
\usepackage{booktabs}
\usepackage[protrusion=true,expansion=false]{microtype}
\usepackage{graphicx}
\usepackage{tikz}
\usetikzlibrary{arrows.meta,calc,decorations.pathreplacing}
\usepackage{subcaption}
\usepackage{geometry}
\usepackage{float}
\usepackage{xcolor}
\newif\ifreview
\reviewtrue 
\ifreview
\definecolor{softpink}{rgb}{1.0,0.55,0.70}
\definecolor{reviewblue}{rgb}{0.05,0.35,0.80}
\newcommand{\problematic}[1]{\begingroup\sloppy\textcolor{red}{#1}\endgroup}

\newcommand{\reviewnote}[1]{\begingroup\sloppy\textit{[Reviewer note: #1]}\endgroup}
\else
\definecolor{softpink}{rgb}{0,0,0}
\definecolor{reviewblue}{rgb}{0,0,0}
\newcommand{\problematic}[1]{}

\newcommand{\reviewnote}[1]{}
\fi
\usepackage[
  colorlinks=true,
  linkcolor=blue,
  citecolor=blue,
  urlcolor=blue
]{hyperref}

\newtheorem{theorem}{Theorem}[section]
\newtheorem{proposition}[theorem]{Proposition}
\newtheorem{lemma}[theorem]{Lemma}
\newtheorem{corollary}[theorem]{Corollary}
\newtheorem{conjecture}[theorem]{Conjecture}
\theoremstyle{definition}
\newtheorem{definition}[theorem]{Definition}
\newtheorem{example}[theorem]{Example}
\theoremstyle{remark}
\newtheorem{remark}[theorem]{Remark}

\newcommand{\phiGR}{\varphi}
\newcommand{\RR}{\mathbb{R}}
\newcommand{\ZZ}{\mathbb{Z}}
\newcommand{\QQ}{\mathbb{Q}}
\newcommand{\Hcech}{\check{H}}
\newcommand{\pE}{E_{\parallel}}     
\newcommand{\pEperp}{E_{\perp}}     

\title{Matching Rules as Cocycle Conditions: Discrete Potentials on Penrose and Canonical Projection Tilings}
\author{
 Sebastian Pardo-Guerra\thanks{Recognition Physics Institute, Austin, TX, USA. \texttt{sebas@recognitionphysics.org}},\quad
 Jonathan Washburn\thanks{Recognition Physics Institute, Austin, TX, USA. \texttt{jon@recognitionphysics.org}},\quad
 Elshad Allahyarov$^{1,2,3,}$\thanks{%
  Recognition Physics Institute, Austin, Texas, USA.\\
  $^{1}$\,Institut f\"ur Theoretische Physik II: Weiche Materie, Heinrich-Heine Universit\"at D\"usseldorf,
  Universit\"atsstra{\ss}e~1, 40225 D\"usseldorf, Germany.\\
  $^{2}$\,Theoretical Department, Joint Institute for High Temperatures, Russian Academy of Sciences (IVTAN),
  13/19 Izhorskaya street, Moscow 125412, Russia.\\
  $^{3}$\,Department of Physics, Case Western Reserve University, Cleveland, Ohio 44106-7202, USA. 
\texttt{elshad.allakhyarov@case.edu}}\quad}
\date{}

\begin{document}
\maketitle

\begin{abstract}
Aperiodic tilings support two classically studied but hitherto separately presented structures: matching rules, which enforce global order via local constraints, and height functions, which encode global geometry through integer-valued potentials. Their precise relationship has remained implicit in the literature. This paper bridges them via a cochain-first framework, establishing a four-way equivalence---between matching rules, Ammann bar continuity, cycle closure of the associated $1$-cochains, and height-function existence---proved for candidate tilings without presupposing any of the four conditions.

The proof proceeds via a half-edge/gluing construction: for each Ammann bar family, we assign to every directed edge a signed bar-crossing count, yielding an antisymmetric $1$-cochain. A tile-side crossing function and a global cochain are built in two stages; the global cochain exists precisely when adjacent tiles agree on shared edges. Gluing implies cycle closure; the discrete Poincar\'{e} lemma then produces a scalar potential coinciding with the classical Ammann height function.

The framework extends uniformly to canonical projection tilings (CPTs) from $\mathbb{Z}^N$: lattice-coordinate cochains reconstruct vertex positions via $v = \sum_{k=1}^N x_k(v)\,\mathbf{e}_k^*$, and (for CPTs with generic window) form a $\mathbb{Z}$-basis for $\check{H}^1 \cong \mathbb{Z}^N$ (Forrest--Hunton--Kellendonk), yielding a conservation-forced structure with recognition gap $\mathcal{R}(\mathcal{T}) \cong \mathbb{Z}^N$. The framework is verified for the Fibonacci chain, Penrose P2, Ammann--Beenker, and the icosahedral Ammann tiling; whether conservation forcing characterises exactly the Pisot substitution CPTs is left as an open conjecture.
\end{abstract}



\medskip
\ifreview
\tableofcontents
\fi
\medskip

\section{Introduction}
\label{sec:introduction}

\subsection{Background and motivation}
\label{sec:background-motivation}

Penrose aperiodic tilings~\cite{Penrose1974,BaakeGrimm2013} are planar tessellations
that exhibit long-range orientational order without translational periodicity.
Their combinatorial structure is enforced by \emph{matching rules}---local
edge-decoration constraints that forbid periodic assemblies---and organised by the
\emph{Ammann height function}~\cite{GrunbaumShephard1987,BaakeGrimm2013},
which assigns integer values to vertices so that height differences encode
geometric information about edge types.

Despite decades of parallel study, these two structures have been developed in entirely separate technical languages: matching rules are treated as combinatorial local constraints, while height functions are constructed globally only after tiling validity is assumed. No unified algebraic framework capable of handling \emph{candidate} tilings---assemblies whose decorations are present but may violate matching rules---has been established. This gap is the central problem the present paper addresses.

This paper closes that gap via a cochain-first framework. The key device is a two-stage half-edge/gluing construction (\S\ref{sec:main}): bar-crossing functions are defined tile-by-tile from local decoration data before any global validity is assumed; a globally consistent $1$-cochain then exists precisely when adjacent tile values agree on every shared directed edge---exactly the matching rule condition. Matching rules are therefore local \emph{conservation} laws: algebraic cycle-closure conditions on a directed graph. The precise four-way equivalence is stated in \S\ref{sec:survey} below.

\subsection{Relationship to prior work}
\label{sec:prior-work}

We now situate the paper in its prior-work context. Understanding what was already known makes the novelty of the unified framework precise.

The correspondence between tiling conditions and discrete $1$-cochains
has roots in several independent lines of work, which together
motivate the unified perspective developed here.

The closest structural antecedent is Thurston~\cite{Thurston1990}, who
showed that tileability of a planar region by dominoes or lozenges is
equivalent to cycle closure of a height-difference cochain on the dual
graph; \S\S\ref{sec:discrete-potential}--\ref{sec:main} are the direct
aperiodic analogue of this framework.
A complementary global approach is provided by
Conway--Lagarias~\cite{ConwayLagarias1990}, who frame tileability as a
word problem in a group; the group identity condition is a conservation
law in the present sense.
A detailed comparison with both frameworks,
including their structural differences from the cochain approach, appears
in \S\ref{sec:discussion}.

The cohomological framework most directly relevant to the present work is
that of Forrest, Hunton, and Kellendonk~\cite{ForrestHuntonKellendonk2002},
who develop \v{C}ech cohomology for canonical projection tilings via
pattern-equivariant cochains and compute
$\check{H}^1(\Omega_P;\ZZ)\cong\ZZ^5$ for the Penrose hull, with the
five bar-crossing cochains as generators.
The present paper provides explicit self-contained graph-theoretic proofs for P2 tilings of the properties that FHK establish at the general CPT level: the half-edge/gluing construction (Definition~\ref{def:half-edge-crossing}, Lemma~\ref{lem:gluing}), face-sum vanishing (Lemmas~\ref{lem:tile-sum-zero} and~\ref{lem:face-sum-general}), and the five-family validity and reconstruction theorem (Theorem~\ref{thm:five-family}) are all made self-contained within this graph-theoretic vocabulary.

A fourth line of work, due to De~Bruijn~\cite{deBruijn1981}, shows
that every P2 tiling arises as the dual of a ``pentagrid''---a
superposition of five families of parallel lines indexed by five
integer strip coordinates.
Theorem~\ref{thm:five-family}(ii) identifies these strip indices
directly with the five Ammann height functions, so that the
potential-theoretic and pentagrid descriptions of vertex position
are unified in a single formula: each vertex is the
$\tfrac{2}{5}$-weighted sum of its five height values along the five
bar normals.
This identification, together with the injectivity of the
height-tuple map $\Phi\colon V\to\ZZ^5$, makes explicit a
consequence that is implicit in the Forrest--Hunton--Kellendonk
framework but not stated there in coordinate form.

For comprehensive background on Penrose tilings, canonical projection
tilings, and their cohomology we recommend
Baake--Grimm~\cite{BaakeGrimm2013},
Gr\"{u}nbaum--Shephard~\cite{GrunbaumShephard1987},
and Sadun~\cite{Sadun2008}.

Taken together, prior work establishes the individual ingredients---height functions, cycle-closure conditions, cohomological generators, and pentagrid coordinates---but stops short of a unified, candidate-tiling-aware cochain language with self-contained graph-theoretic proofs and a direct extension to general canonical projection tilings. The present paper fills this gap by building that language from the ground up.

\subsection{Survey scope and central correspondence}
\label{sec:survey}

With the prior work in view, we now describe the main results and the scope of the unified framework developed in this paper.

The individual ingredients---height functions, matching rules, the discrete
Poincar\'{e} lemma, bar continuity, and the cohomology of tiling spaces---are
all known (see \S\ref{sec:prior-work}).
This paper assembles these ingredients into a single cochain/potential framework with uniform notation, clearly separating general facts about $1$-cochains on graphs from results specific to Penrose decorations and Ammann bar data; a taxonomy of what is classical, reorganised, or new appears in \S\ref{sec:novelty}.

Theorem~\ref{thm:matching-conservation} establishes the following four-way equivalence for candidate tilings (Proposition~\ref{prop:height-potential} identifies the individual bar-crossing cochains as scalar potentials):

\begin{enumerate}[label=(\roman*),noitemsep]
\item $\mathcal{T}$ satisfies the Penrose matching rules.
\item Ammann bars form continuous straight lines across $\mathcal{T}$.
\item The bar-crossing functions $\Delta_{\mathbf{n}}$ satisfy cycle closure.
\item A globally consistent height function $h\colon V \to \ZZ$ exists.
\end{enumerate}

The equivalence \textup{(i)}$\Leftrightarrow$\textup{(ii)} (matching rules $\Leftrightarrow$ bar continuity) is classical~\cite{GrunbaumShephard1987,BaakeGrimm2013}; the remaining three implications are new and are proved here via the half-edge/gluing construction (Definition~\ref{def:half-edge-crossing}, Lemma~\ref{lem:gluing}).

Table~\ref{tab:dictionary} summarises the structural correspondence developed in \S\S\ref{sec:discrete-potential}--\ref{sec:main}:

\begin{table}[ht]
\centering
\caption{Structural correspondence between Penrose tiling theory
and discrete potential theory.}
\label{tab:dictionary}
\small
\begin{tabular}{ll}
\toprule
\textbf{Penrose Tiling Theory} & \textbf{Discrete Potential Theory} \\
\midrule
Tiling adjacency graph $G_{\mathcal{T}}$ & Connected graph $G$ \\
Edge with bar-crossing data & Directed edge with antisymmetric function $\Delta$ \\
Bar continuity (from matching rules) & Cycle closure \\
Ammann height function $h$ & Scalar potential $h$ \\
Matching rule at shared edge & Antisymmetric conservation \\
\bottomrule
\end{tabular}
\end{table}

\subsection{Scope of contributions}
\label{sec:novelty}

For transparency, the paper's content falls into three categories,
surveyed here in increasing order of novelty: classical results recalled
for self-containedness, known results reorganised into unified
cochain-potential language, and several results and constructions proved here.

Several classical facts are recalled without re-proof to keep the
paper self-contained.
It is known that height functions for valid Penrose tilings exist and
coincide with the Ammann
height~\cite{GrunbaumShephard1987,BaakeGrimm2013}, that matching rules
are equivalent to height-function existence~\cite{BaakeGrimm2013}, and
that the substitution entropy of any primitive substitution equals
$\ln\lambda_{\mathrm{PF}}$~\cite{Queffelec2010}.
Thurston's height-function framework~\cite{Thurston1990} establishes the
cycle-closure--height-function equivalence in the periodic setting that
the present paper imports into the aperiodic regime; at the cohomological
level, the computation
$\check{H}^1(\Omega_P;\ZZ)\cong\ZZ^5$ with generators given by the five
bar-crossing cochains is due to
Forrest--Hunton--Kellendonk~\cite{ForrestHuntonKellendonk2002}; and
de~Bruijn's pentagrid construction~\cite{deBruijn1981} identifies vertex
positions via five integer strip coordinates.

A second layer of the paper takes these ingredients and reorganises them
into a coherent cochain-potential language with unified notation.
The key organisational contribution is the two-stage gluing construction
(Definition~\ref{def:half-edge-crossing}, Lemma~\ref{lem:gluing}), which
defines bar-crossing functions tile-by-tile \emph{before} assuming global
validity; this device allows Proposition~\ref{prop:height-potential} and
Theorem~\ref{thm:matching-conservation} to package the classical
equivalences for candidate tilings in a non-circular way, since the
global cochain need not exist a priori.
Theorem~\ref{thm:five-family} then assembles all five height functions into a single joint reconstruction formula that identifies de~Bruijn's strip indices with the Ammann heights, providing a self-contained unified notation for both descriptions.
The conservation-law reading of matching rules follows from applying the
cycle-closure characterisation of
Conway--Lagarias~\cite{ConwayLagarias1990} to the P2 setting, and the
recognition gap (\S\ref{sec:recognition-gap}) names the standard
cohomological observation that $H^1(X_{\mathcal{T}};\ZZ)=0$ while
$H^1_{PE}\cong\ZZ^5$, making precise the sense in which each bar-crossing
cochain is locally readable yet globally non-trivial.
Corollary~\ref{cor:standard-cf} for the Ammann--Beenker tiling provides a
second independent example confirming that the same cochain vocabulary
extends beyond Penrose without modification.

The main contributions of this paper may be summarised as follows.
First, it gives a self-contained cochain-first formulation of Penrose
matching rules, built to handle \emph{candidate} (possibly invalid) decorated
tilings rigorously via the half-edge/gluing construction
(Definition~\ref{def:half-edge-crossing}, Lemma~\ref{lem:gluing}).
Second, it packages the classical Penrose correspondences into explicit
graph-theoretic statements: the four-way equivalence for candidate tilings
(Theorem~\ref{thm:matching-conservation}) and the five-family validity and
reconstruction theorem (Theorem~\ref{thm:five-family}) in a unified notation.
Third, it introduces the conservation-forced framework
(Definition~\ref{def:conservation-forced}) and proves that primitive-substitution
CPTs are conservation-forced with conservation rank equal to the ambient lattice
dimension (Theorem~\ref{thm:cpt-conservation-forced}), thereby establishing the
implication \textup{(b)$\Rightarrow$(a)} of Conjecture~\ref{conj:cf} for the CPT
class.

\subsection{Organization}
\label{sec:organization}

Section~\ref{sec:background} establishes the background on Penrose tilings
and canonical projection tilings, including candidate decorated tilings and
local Ammann bar data.
Section~\ref{sec:discrete-potential} develops discrete potential theory on
graphs.
Section~\ref{sec:main} builds the full tiling--potential correspondence:
the half-edge/gluing construction, the two central equivalence results
(Proposition~\ref{prop:height-potential} and Theorem~\ref{thm:matching-conservation}),
the five-family validity and pentagrid reconstruction theorem,
the cohomological interpretation, and a validation remark.
Section~\ref{sec:conservation-forced} introduces the notion of a
\emph{conservation-forced} tiling system, proves the general CPT theorem
(Theorem~\ref{thm:cpt-conservation-forced}), verifies the definition for
Penrose P2 (rank~$5$), Ammann--Beenker (rank~$4$), the Fibonacci chain
(rank~$2$), and the icosahedral tiling (rank~$6$), states the recognition gap
$\mathcal{R}(\mathcal{T})\cong\ZZ^N$, and poses the classification conjecture.
Section~\ref{sec:discussion} provides a detailed comparison
with Thurston~\cite{Thurston1990} and
Conway--Lagarias~\cite{ConwayLagarias1990}, and lists open problems.
Section~\ref{sec:conclusion} concludes.
Appendix~\ref{app:tile-calcs} gives explicit tile-boundary calculations.
Appendix~\ref{app:golden} records the role of $\phiGR$ as the Perron--Frobenius eigenvalue of the Robinson substitution, its connection to the substitution-entropy picture, and a coherence-hierarchy ordering of the four verified CPT families by their Perron--Frobenius eigenvalues.

\section{Background: Penrose Tilings and Canonical Projection Tilings}
\label{sec:background}

This section develops the background material for Sections~\ref{sec:discrete-potential}--\ref{sec:conservation-forced}.
It is organised as follows.
\S\ref{sec:assumptions} fixes the standing conventions used throughout the paper.
\S\ref{sec:penrose-background} recalls the Penrose P2 dart--kite prototiles, the matching rules, Ammann bars, and the Ammann height function.
\S\ref{sec:candidate-tiling} introduces candidate decorated tilings, formalises local bar data for tilings that may violate the matching rules, and defines the tiling adjacency graph~$G_{\mathcal{T}}$ that carries all cochains of Section~\ref{sec:discrete-potential}.
\S\ref{sec:standing-facts} collects the geometric properties of Ammann bars used in the half-edge/gluing construction of Section~\ref{sec:main}.
\S\ref{sec:cpt-background} extends the framework to canonical projection tilings, setting up the conservation-forced analysis of Section~\ref{sec:conservation-forced}.

Before introducing the objects studied in this section, we fix the standing conventions and scope restrictions used throughout the paper.

\subsection{Standing assumptions and scope conventions}
\label{sec:assumptions}

To prevent hidden hypotheses
from being used implicitly, we adopt the following standing conventions.
\begin{enumerate}[label=\textup{(\alph*)},noitemsep]
\item \textbf{Penrose model.} Unless stated otherwise, ``Penrose tiling'' means the P2
(dart--kite) system with the standard edge decorations and Ammann bars as in
\S\ref{sec:penrose-background}.
\item \textbf{Candidate vs.\ valid.} ``Candidate decorated tiling'' means Definition~\ref{def:candidate-tiling}
(decorations are present but may violate matching rules). Statements that require validity
explicitly assume matching rules (or invoke Theorem~\ref{thm:matching-conservation}).
\item \textbf{Graph conventions.} All cochains on tilings are functions on the directed-edge set
\(\vec{E}\) of the tiling adjacency graph \(G_{\mathcal{T}}\) (Definition~\ref{def:tiling-recog-graph}).
\item \textbf{Cut-and-project meaning.} ``Canonical projection tiling (CPT)'' always means
Definition~\ref{def:cpt}, including \emph{injectivity} and the \emph{generic window} condition.
When we appeal to the Forrest--Hunton--Kellendonk rank/basis result, the generic-window
hypothesis is the relevant input.
\item \textbf{Substitution scope.} The general conservation-forced theorem
(Theorem~\ref{thm:cpt-conservation-forced}) is proved for \emph{primitive-substitution} CPTs.
Whenever ``Pisot'' is mentioned, it is as a potential classification hypothesis in the open
direction (Conjecture~\ref{conj:cf}), not as an input to the proved direction.
\end{enumerate}

\subsection{Prototiles, matching rules, and Ammann bars}
\label{sec:penrose-background}

We work with the dart--kite (P2) formulation~\cite{Penrose1974,BaakeGrimm2013},
where each prototile is a quadrilateral with edges of length~$1$ (short) and
$\phiGR = (1+\sqrt{5})/2$ (long).
The Robinson triangle decomposition~\cite{Robinson1975,BaakeGrimm2013} splits
each dart or kite into golden triangles~$T$
(isosceles, $36^\circ$-$72^\circ$-$72^\circ$) and golden gnomons~$G$
(isosceles, $108^\circ$-$36^\circ$-$36^\circ$).
These two similarity types satisfy the primitive Robinson triangle substitution
\begin{equation}
  T \;\to\; T + G, \qquad G \;\to\; T,
\label{eq:substitution}
\end{equation}
whose tile-count dynamics are encoded by the Fibonacci matrix
\begin{equation}
  M = \begin{pmatrix} 1 & 1 \\ 1 & 0 \end{pmatrix},
  \quad \text{with eigenvalues}\quad
  \lambda_1 = \phiGR,\quad \lambda_2 = -1/\phiGR.
\label{eq:subst-matrix}
\end{equation}
The Perron--Frobenius eigenvalue $\phiGR$ governs the exponential growth of the tile count under iterated substitution: after $n$ steps the total grows as $\phiGR^n$.
The substitution structure and the associated matrix~$M$ are background material here; they become essential in Section~\ref{sec:conservation-forced}, where the Perron--Frobenius eigenvalue is the key invariant distinguishing conservation-forced tiling systems.

Each prototile edge carries arrow and arc decorations enforcing the
\emph{matching rules}: two tiles may share an edge only if their decorations
agree, and these local constraints forbid periodic arrangements and enforce
global aperiodicity~\cite{Penrose1974}.

For a valid tiling, the decorations also determine five families of
\emph{Ammann bars}---sets of parallel straight lines crossing the tiling
continuously~\cite{GrunbaumShephard1987,BaakeGrimm2013}.
Each family has a fixed direction (one of five related by $72^\circ$ rotations, illustrated in Figure~\ref{fig:five-normals}),
and the perpendicular spacings between consecutive bars form a Fibonacci
quasiperiodic sequence with short gap $S$ and long gap $L = \phiGR\,S$.

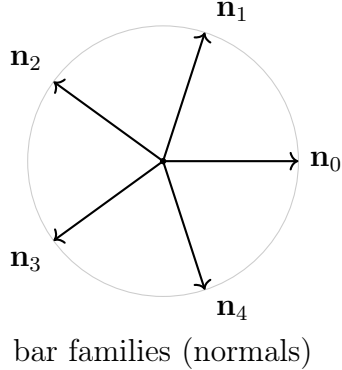
\begin{figure}[H]
\centering
\begin{tikzpicture}[scale=1.0]
  \draw[gray!40] (0,0) circle (1.8);
  \draw[->,thick] (0,0) -- (0:1.8) node[anchor=west] {$\mathbf{n}_0$};
  \draw[->,thick] (0,0) -- (72:1.8) node[anchor=south west] {$\mathbf{n}_1$};
  \draw[->,thick] (0,0) -- (144:1.8) node[anchor=south east] {$\mathbf{n}_2$};
  \draw[->,thick] (0,0) -- (216:1.8) node[anchor=north east] {$\mathbf{n}_3$};
  \draw[->,thick] (0,0) -- (288:1.8) node[anchor=north west] {$\mathbf{n}_4$};
  \fill (0,0) circle (1.2pt);
  \node[below] at (0,-2.2) {bar families (normals)};
\end{tikzpicture}
\caption{The five Ammann bar families for Penrose P2 can be indexed by
unit normals $\mathbf{n}_k$ related by $72^\circ$ rotation.  A ``bar family
direction'' means the direction of the bar lines; its normal $\mathbf{n}_k$
is used to assign signed crossing numbers to directed edges.}
\label{fig:five-normals}
\end{figure}

A classical result~\cite{BaakeGrimm2013} establishes that the matching rules
are satisfied if and only if Ammann bars form continuous straight lines
(bars cannot terminate inside the tiling; see Definition~\ref{def:local-global-bars}
below for the formal definition of bar continuity).
Fixing one bar family with unit normal $\mathbf{n}$, the
\emph{Ammann height function} $h\colon V \to \ZZ$ assigns to each vertex $v$
the signed number of bars between a reference vertex $v_0$ and $v$, counted
with respect to $\mathbf{n}$~\cite{GrunbaumShephard1987}; the height difference
across each edge encodes whether the edge crosses a bar and in which direction.

\subsection{Candidate decorated tilings and local bar data}
\label{sec:candidate-tiling}

The following definitions make bar data meaningful for
\emph{candidate} tilings that may violate the matching rules---providing the
vocabulary needed for Theorem~\ref{thm:matching-conservation} and
Remark~\ref{rem:algorithm}---and introduce the adjacency graph that carries
the cochains of Section~\ref{sec:discrete-potential}.

\begin{definition}[Candidate decorated Penrose tiling]
\label{def:candidate-tiling}
A \emph{candidate (decorated) Penrose tiling} is an edge-to-edge tiling
$\mathcal{T}$ of $\RR^2$ by dart and kite prototiles, where each tile carries the
standard Penrose edge decorations (arrows/arcs), but adjacent tiles are not
assumed to satisfy the matching rules (decorations on a shared edge may disagree).
\end{definition}

\begin{definition}[Local and global Ammann bars]
\label{def:local-global-bars}
Fix one of the five Ammann bar family directions, specified by a unit normal
vector $\mathbf{n}$.
Each decorated prototile determines a finite collection of straight \emph{local bar segments} inside the tile, all lying on lines with normal $\mathbf{n}$ (equivalently, all segments are parallel to the bar direction)~\cite{GrunbaumShephard1987,BaakeGrimm2013}.
For a candidate tiling $\mathcal{T}$, let $\mathcal{B}_{\mathbf{n}}(\mathcal{T})$
denote the set of connected components of the union of all such segments over
tiles of $\mathcal{T}$.
Elements of $\mathcal{B}_{\mathbf{n}}(\mathcal{T})$ are the (possibly broken)
\emph{Ammann bars} of family~$\mathbf{n}$.
We say that the $\mathbf{n}$-bar family is \emph{continuous} if every component
of $\mathcal{B}_{\mathbf{n}}(\mathcal{T})$ is an entire straight line in~$\RR^2$
(equivalently, has no endpoints in the interior of~$\RR^2$).
\end{definition}

The candidate tiling and its bar data live on a natural combinatorial structure: tiles that share a tile edge are declared adjacent, giving the following graph.

\phantomsection\label{sec:recog-graph}
\begin{definition}[Tiling adjacency graph]
\label{def:tiling-recog-graph}
Let $\mathcal{T}$ be a Penrose tiling with vertex set $V$ and edge set~$E$.
The \emph{tiling adjacency graph} is $G_{\mathcal{T}} = (V,E)$, where
$(u,v) \in E$ whenever $u$ and $v$ are connected by a tile edge.
\end{definition}

\begin{remark}[Basic graph properties]
\label{rem:graph-properties}
For an edge-to-edge locally finite tiling $\mathcal{T}\subset\RR^2$ by polygonal
tiles, the adjacency graph $G_{\mathcal{T}}$ is connected and planar: it is the
$1$-skeleton of the tiling CW complex embedded in the plane.  See, e.g.,
Gr\"{u}nbaum--Shephard~\cite{GrunbaumShephard1987} or Baake--Grimm~\cite{BaakeGrimm2013}.
\end{remark}

\subsection{Geometric properties of Ammann bars}
\label{sec:standing-facts}

The following standard properties of Penrose P2 decorations and Ammann bars
are used in \S\ref{sec:main}; we refer to
Baake--Grimm~\cite[Ch.~6]{BaakeGrimm2013} for concrete conventions and
diagrams.

\begin{lemma}[Local bar segments are decoration-determined]
\label{lem:local-bars-determined}
Fix a bar-family direction (equivalently, a bar normal $\mathbf{n}$).
For each oriented Penrose prototile, the standard edge decorations
determine a finite set of straight \emph{local} Ammann bar segments inside
the tile, in the sense of Definition~\ref{def:local-global-bars}.  In
particular, for a candidate decorated tiling (Definition~\ref{def:candidate-tiling}),
local bar data are readable tile-by-tile without any global validity
assumption.
\end{lemma}
\begin{proof}
By Definition~\ref{def:local-global-bars}, each decorated prototile determines
a finite collection of local bar segments from its decoration data alone, with
no reference to adjacent tiles.
Since a candidate tiling assigns decorations tile-by-tile
(Definition~\ref{def:candidate-tiling}), the local bar segments of every tile
are defined independently of whether the global matching rules are satisfied.
\end{proof}

\begin{remark}[Transversality and uniqueness properties]
\label{rem:transversality}
For Penrose P2 tilings with the standard Ammann bar directions, the
following geometric properties hold (see~\cite[Ch.~6]{BaakeGrimm2013}):
\begin{enumerate}[label=(\roman*),noitemsep]
\item No bar of a given family is parallel to any tile edge, and no bar
      passes through a tiling vertex.
\item For each tile and each bar family, each tile edge is crossed by at
      most one local bar segment of that family.
\item The five bar families correspond to the five edge-direction classes of P2: each edge direction class determines a unique bar family whose bars cross edges of that direction.
\item For the unique relevant family associated with an edge-direction
      class, every edge of that class is crossed by a local bar segment of
      that family (hence the relevant tile-side crossing value is always
      $\pm 1$, never $0$).
\end{enumerate}
These properties justify the $\{-1,0,+1\}$-valued crossing conventions in
Definition~\ref{def:half-edge-crossing} and the ``unique relevant family''
language in Lemma~\ref{lem:gluing}.
\end{remark}

\subsection{Canonical projection tilings}
\label{sec:cpt-background}

The potential-theoretic framework extends beyond Penrose tilings to the
entire class of \emph{canonical projection tilings (CPTs)}; we introduce
this class here to set up Section~\ref{sec:conservation-forced}.

\begin{definition}[Canonical projection tiling]
\label{def:cpt}
Let $N > d \geq 1$ be integers, $\Lambda \subset \RR^N$ a full-rank
lattice (usually $\ZZ^N$), and $\pE \subset \RR^N$ a $d$-dimensional
subspace (the \emph{physical space}) such that the projection
$\pi\colon\RR^N \to \pE$ satisfies:
\begin{enumerate}[label=(\roman*),noitemsep]
\item \emph{Irrationality:} $\pi(\Lambda)$ is dense in $\pE$.
\item \emph{Injectivity:} $\pi|_{\Lambda}$ is injective (equivalently,
      $\Lambda \cap \pEperp = \{\mathbf{0}\}$).
\item \emph{Generic window:} the acceptance window
      $W\subset\pEperp := (\pE)^\perp$ is a compact polytope whose
      boundary meets no projected lattice point, i.e.\
      $\partial W \cap \pi_\perp(\Lambda) = \varnothing$.
\end{enumerate}
The \emph{vertex set} of the CPT is
$V = \{ \pi(\mathbf{x}) : \mathbf{x}\in\Lambda,\; \pi_\perp(\mathbf{x})\in W \}$,
where $\pi_\perp\colon\RR^N\to\pEperp$ is the complementary projection.
The tiles are the images under $\pi$ of the Voronoi cells of $\Lambda$
whose $\pi_\perp$-projections intersect~$W$.
\end{definition}

The five standard CPT families studied in this paper are collected
in Table~\ref{tab:cpt-families}.

\begin{table}[ht]
\centering
\caption{Standard CPT families.  Here $N$ is the ambient lattice
  dimension and $d$ the physical dimension.
  The column $\lambda_{\mathrm{PF}}$ is the Perron--Frobenius eigenvalue of the substitution.}

\label{tab:cpt-families}
\small
\begin{tabular}{llccc}
\toprule
\textbf{Tiling} & \textbf{Prototiles} & $d$ & $N$ & $\lambda_{\mathrm{PF}}$ \\
\midrule
Fibonacci chain         & S, L segments          & 1 & 2 & $\phiGR \approx 1.618$ \\
Penrose P2/P3           & dart--kite / thick--thin rhombi & 2 & 5 & $\phiGR \approx 1.618$ \\
Ammann--Beenker         & square, $45^\circ$ rhombus & 2 & 4 & $1+\sqrt{2} \approx 2.414$ \\
Icosahedral Ammann      & two rhombohedra        & 3 & 6 & $\phiGR \approx 1.618$ \\
\bottomrule
\end{tabular}
\end{table}

\begin{definition}[Lattice-coordinate map]
\label{def:lattice-coord}
For a CPT as in Definition~\ref{def:cpt}, the \emph{lattice-coordinate map}
is
\[
  \Phi\colon V \longrightarrow \ZZ^N,\qquad
  \Phi(v) := \mathbf{x}(v),
\]
where $\mathbf{x}(v)\in\Lambda$ is the unique lattice point with
$\pi(\mathbf{x}(v)) = v$ and $\pi_\perp(\mathbf{x}(v))\in W$.
Uniqueness follows from injectivity of $\pi|_{\Lambda}$ in
Definition~\ref{def:cpt}.
For the standard lattice $\Lambda = \ZZ^N$, write
$\Phi(v) = (x_1(v),\ldots,x_N(v))$, and define the
\emph{$k$-th bar-crossing cochain} by
\begin{equation}
  \Delta_k(u\to v) \;:=\; x_k(v) - x_k(u),
  \qquad k = 1,\ldots,N.
\label{eq:lattice-cochain}
\end{equation}
\end{definition}

\begin{remark}[Gradient structure]
\label{rem:gradient-cpt}
Since $\Delta_k$
is defined as a coordinate difference $x_k(v)-x_k(u)$, any sum around
a closed loop telescopes to zero; Lemma~\ref{lem:face-sum-general} makes this precise.
\end{remark}

\begin{lemma}[Generalised face-sum vanishing for CPT $2$-cells]
\label{lem:face-sum-general}
Let $\mathcal{T}$ be a CPT in $\RR^d$ (Definition~\ref{def:cpt}) arising from
a lattice $\ZZ^N$, and let $\Delta_k$ be the $k$-th bar-crossing cochain
defined by~\eqref{eq:lattice-cochain}.
Then for every $2$-cell $f$ of the tiling CW complex $X_{\mathcal{T}}$ (equivalently,
for every polygonal $2$-face occurring in $\mathcal{T}$) and every $k\in\{1,\ldots,N\}$,
\begin{equation}
  \sum_{(u\to v)\in\partial f} \Delta_k(u\to v) = 0,
\label{eq:face-sum-general}
\end{equation}
where $\partial f$ is traversed with its induced boundary orientation.
\end{lemma}
\begin{proof}
By Remark~\ref{rem:gradient-cpt}, $\Delta_k(u\to v) = x_k(v)-x_k(u)$ is a coordinate difference.
The boundary sum therefore telescopes:
\[
  \sum_{(u\to v)\in\partial f} \Delta_k(u\to v)
  = \sum_{(u\to v)\in\partial f} \bigl[x_k(v)-x_k(u)\bigr] = 0,
\]
since each vertex of $f$ appears once as a head and once as a tail as $\partial f$
is traversed.
\end{proof}

\begin{remark}[Comparison with the Penrose geometric argument]
\label{rem:comparison-tile-sum}
Lemma~\ref{lem:face-sum-general} gives a purely algebraic proof of face-sum
vanishing that holds for all CPT dimensions and all lattice-coordinate cochains,
showing the cocycle condition is a \emph{formal consequence} of the gradient
structure $\Delta_k = \delta_0 x_k$ rather than a special geometric property of
the prototile family.
The same vanishing for Penrose P2 is established by a geometric transversality
argument in Lemma~\ref{lem:tile-sum-zero} (\S\ref{sec:bar-postings}).
\end{remark}

\begin{lemma}[Local determinacy of lattice steps]
\label{lem:cpt-local-step}
Let $\mathcal{T}$ be a CPT with \emph{finite local complexity} (FLC: up to translation, only finitely many distinct radius-$R$ patches exist for each $R > 0$), with lattice-coordinate map $\Phi\colon V\to\ZZ^N$ (Definition~\ref{def:lattice-coord}).
Then, there exists $R>0$ such that for every directed edge $(u\to v)$ in
$G_{\mathcal{T}}$, the lattice difference vector $\Phi(v)-\Phi(u)\in\ZZ^N$
is determined by the radius-$R$ patch of $\mathcal{T}$ around the directed
edge $(u\to v)$.  In particular, for each $k=1,\ldots,N$, the cochain
$\Delta_k(u\to v)=x_k(v)-x_k(u)$ is pattern-equivariant.
\end{lemma}

\begin{proof}
We show that the lattice step $\Phi(v)-\Phi(u)\in\ZZ^N$ across any
directed edge $(u\to v)$ is determined by a bounded neighbourhood of
the edge.

\emph{Step~1 (finitely many lattice steps).}
By Definition~\ref{def:cpt}, each tile
$t\in\mathcal{T}$ is the image under~$\pi$ of a $d$-dimensional face of
a Voronoi cell of $\Lambda=\ZZ^N$ in~$\RR^N$.
If $t = \pi(F)$ for a Voronoi face~$F$ with vertices
$\mathbf{x}_0,\ldots,\mathbf{x}_m\in\ZZ^N$, then the vertices of~$t$
are $\pi(\mathbf{x}_0),\ldots,\pi(\mathbf{x}_m)$, and each edge of~$t$
has the form $\{\pi(\mathbf{x}_i),\pi(\mathbf{x}_j)\}$ where
$\mathbf{x}_j - \mathbf{x}_i$ is a \emph{Voronoi-adjacency vector}: a
difference between two lattice points whose Voronoi cells share a face.
The set of Voronoi-adjacency vectors of~$\ZZ^N$ is a fixed finite set~$S$
depending only on the lattice (not on the tiling), so
$\Phi(v)-\Phi(u)\in S$ for every directed edge $(u\to v)$.

\emph{Step~2 (prototile type determines lattice step).}
Each prototile type of~$\mathcal{T}$ is the $\pi$-image of a specific
Voronoi face type~$F_\alpha$ of~$\ZZ^N$
(there are finitely many such types because~$\ZZ^N$ has finitely many
Voronoi-face classes up to lattice translation).
The boundary edges of~$F_\alpha$ carry specific Voronoi-adjacency
vectors from~$S$; hence, for each prototile type and each edge-position
within that prototile, the lattice step $\Phi(v)-\Phi(u)$ is a fixed
element of~$S$.

\emph{Step~3 (local patch determines prototile and edge-position).}
Finite local complexity implies that, up to translation, there are only
finitely many distinct radius-$R$ edge-centred patches for any
fixed~$R$.
Choose $R$ larger than the diameter of any prototile.
Then the radius-$R$ patch around a directed edge $(u\to v)$ determines
which prototile type(s) are incident to the edge and which boundary
edge of each incident prototile $\{u,v\}$ corresponds to.
By Step~2, this determines $\Phi(v)-\Phi(u)$.

\emph{Step~4 (pattern-equivariance).}
If two directed edges $(u\to v)$ and $(u'\to v')$ have
translation-equivalent radius-$R$ patches, they occupy the same
edge-position within the same prototile type, so Steps~2--3 give
$\Phi(v)-\Phi(u) = \Phi(v')-\Phi(u')$.
Therefore each coordinate difference
$\Delta_k(u\to v) = x_k(v)-x_k(u)$ depends only on the radius-$R$
patch, i.e., is pattern-equivariant.
\end{proof}

\begin{proposition}[Reconstruction formula for CPTs]
\label{prop:cpt-reconstruction}
Let $\mathcal{T}$ be a CPT in $\RR^d$ arising from a lattice $\ZZ^N$
with an orthonormal basis $\{\mathbf{e}_j\}_{j=1}^N$ of $\RR^N$.
Write $\mathbf{e}_j^* := \pi(\mathbf{e}_j) \in \RR^d$ for the projected
basis vectors.  Then for every vertex $v\in V$,
\begin{equation}
  v \;=\; \sum_{k=1}^N x_k(v)\,\mathbf{e}_k^*,
\label{eq:cpt-reconstruction}
\end{equation}
where $x_k(v)$ is the $k$-th lattice coordinate of $\Phi(v)$.
In particular, $v$ is determined by the $N$ integer height functions
$h_k := x_k\colon V \to \ZZ$.
\end{proposition}

\begin{proof}
By Definition~\ref{def:lattice-coord},
$\Phi(v) = \sum_{k=1}^N x_k(v)\,\mathbf{e}_k \in \ZZ^N \subset \RR^N$.
Applying $\pi$:
\[
  v = \pi(\Phi(v)) = \pi\!\Bigl(\sum_{k=1}^N x_k(v)\,\mathbf{e}_k\Bigr)
  = \sum_{k=1}^N x_k(v)\,\pi(\mathbf{e}_k) = \sum_{k=1}^N x_k(v)\,\mathbf{e}_k^*. \qedhere
\]
\end{proof}

\begin{remark}[Penrose specialisation]
\label{rem:penrose-specialisation}
For Penrose P2/P3 one has $d=2$, $N=5$.
Write $\mathbf{n}_k = (\cos(2\pi k/5),\,\sin(2\pi k/5))$ for the five
unit normals to the Ammann bar families.
The physical-space projection $\pi\colon\RR^5\to\RR^2$ in the standard
Penrose cut-and-project scheme sends the standard basis vectors to
\[
  \mathbf{e}_k^*
  \;=\; \pi(\mathbf{e}_k)
  \;=\; \tfrac{2}{5}\,\mathbf{n}_k
  \;=\; \tfrac{2}{5}\bigl(\cos(2\pi k/5),\;\sin(2\pi k/5)\bigr),
  \qquad k = 0,\ldots,4.
\]
The factor $\tfrac{2}{5}$ arises because the five unit normals satisfy
the resolution-of-identity
$\sum_{k=0}^{4}(\mathbf{n}_k \otimes \mathbf{n}_k) = \tfrac{5}{2}\,I_2$
(a fact verified in Theorem~\ref{thm:five-family}(ii) below):
the linear map $\RR^5\to\RR^2$ given by
$\mathbf{x}\mapsto\sum_k x_k\,\mathbf{n}_k$ has the Gram matrix
$\tfrac{5}{2}\,I_2$, so the metrically correct projection carries the
reciprocal factor~$\tfrac{2}{5}$.
Equation~\eqref{eq:cpt-reconstruction} then gives
\[
  v \;=\; \sum_{k=0}^{4} x_k(v)\,\mathbf{e}_k^*
  \;=\; \frac{2}{5}\sum_{k=0}^{4} h_k(v)\,\mathbf{n}_k,
\]
recovering de~Bruijn's pentagrid formula~\cite{deBruijn1981}.
\end{remark}

With the CPT framework in place, Section~\ref{sec:discrete-potential} develops the abstract algebraic machinery---antisymmetric edge functions, cycle closure, and the discrete Poincar\'{e} lemma---that will be specialised to Penrose tilings in Section~\ref{sec:main} and to general CPTs in Section~\ref{sec:conservation-forced}.

\section{Discrete Potential Theory on Graphs}
\label{sec:discrete-potential}

This section establishes the abstract algebraic backbone used in Sections~\ref{sec:main} and~\ref{sec:conservation-forced}.
Working with an arbitrary connected graph---so that the theory applies both to Penrose tilings in Section~\ref{sec:main} and to canonical projection tilings in Section~\ref{sec:conservation-forced}---we introduce antisymmetric edge functions ($1$-cochains), the cycle-closure condition, and prove the discrete Poincar\'{e} lemma (Theorem~\ref{thm:poincare}): a $1$-cochain arises as the gradient of a scalar potential if and only if it is cycle-closed.
As will be shown in Section~\ref{sec:main}, the signed Ammann bar-crossing count on the tiling adjacency graph $G_{\mathcal{T}}$ (Definition~\ref{def:tiling-recog-graph}) is precisely such a $1$-cochain, and cycle closure is equivalent to the matching rules.
The potential is the classical Ammann height function.

The following definitions are standard in combinatorial Hodge theory and
graph cohomology; see Biggs~\cite{Biggs1974} and (for the tiling-height-function
viewpoint) Thurston~\cite{Thurston1990}.

\begin{definition}[Antisymmetric edge function]
\label{def:antisymmetric}
Let $G = (V,E)$ be a connected graph with vertex set $V$ and (undirected)
edge set $E$.  Write
\[
  \vec{E} := \{\, (u \to v) \;:\; \{u,v\} \in E \,\}
\]
for the set of directed edges (each undirected edge contributes two orientations).
An \emph{antisymmetric edge function} on~$G$ is a function
$\Delta\colon \vec{E} \to \RR$ such that
\begin{equation}
  \Delta(v \to u) = -\,\Delta(u \to v)
  \quad\text{for all } \{u,v\} \in E.
\label{eq:antisymmetry}
\end{equation}
\end{definition}

\begin{remark}[$1$-cochain terminology]
\label{rem:cochain-terminology}
An antisymmetric edge function on $G$ is also called a \emph{$1$-cochain} (or \emph{antisymmetric $1$-cochain}) on $G$: it is precisely a function on oriented edges satisfying the sign-reversal condition~\eqref{eq:antisymmetry}, i.e., an element of the cellular cochain group $C^1(G;\RR)$.
The cycle-closure condition (Definition~\ref{def:cycle-closure} below) is then exactly the condition that $\Delta$ is a \emph{cocycle}---$\delta^1 \Delta = 0$ in the cellular coboundary sequence---making the discrete Poincar\'{e} lemma (Theorem~\ref{thm:poincare}) a statement about the first cohomology group $H^1(G;\RR)$ being trivial for connected graphs.
This is the algebraic identity underlying the paper's central claim that matching rules are cocycle conditions.
\end{remark}

\begin{definition}[Cycle closure]
\label{def:cycle-closure}
An antisymmetric edge function $\Delta$ satisfies \emph{cycle closure} if for every
directed closed walk $\gamma = (v_0 \to v_1 \to \cdots \to v_k = v_0)$ in~$G$
(vertices may repeat),
\begin{equation}
  \sum_{i=0}^{k-1} \Delta(v_i \to v_{i+1}) = 0.
\label{eq:cycle-closure}
\end{equation}
\end{definition}

\begin{example}[Toy illustration of cycle closure and potentials]
\label{ex:toy-potential}
Consider a square graph with four vertices $a,b,c,d$ (in order) and
edges $\{a,b\},\{b,c\},\{c,d\},\{d,a\}$.  Define the antisymmetric
edge function
\[
  \Delta(a\to b)=+1,\quad \Delta(b\to c)=0,\quad
  \Delta(c\to d)=-1,\quad \Delta(d\to a)=0.
\]
The closed walk $a\to b\to c\to d\to a$ sums to
$(+1)+0+(-1)+0=0$: cycle closure holds.
Setting $h(a)=0$ one can recover $h(b)=1$, $h(c)=1$, $h(d)=0$
so that $\Delta(u\to v)=h(v)-h(u)$ on every edge.
In \S\ref{sec:main} the same arithmetic will appear with
$\Delta = \Delta_{\mathbf{n}}$ (the bar-crossing function)
and $h = h_{\mathbf{n}}$ (the Ammann height).
The following theorem formalises this.
\end{example}

\begin{theorem}[Discrete Poincar\'{e} lemma]
\label{thm:poincare}
Let $\Delta$ be an antisymmetric edge function on a connected graph $G=(V,E)$.
Then $\Delta$ satisfies cycle closure if and only if there exists a function
$h\colon V \to \RR$ (unique up to an additive constant) such that
\begin{equation}
  \Delta(u \to v) = h(v) - h(u) \quad\text{for all } (u \to v) \in \vec{E}.
\label{eq:potential}
\end{equation}
Moreover, if $\Delta(\vec{E})\subseteq\ZZ$ and one fixes $h(v_0)=0$
at a reference vertex, then the path-integral construction yields
$h\colon V\to\ZZ$.
\end{theorem}

\begin{proof}
The result is standard in combinatorial Hodge theory; we include the proof for completeness, since the explicit path-integral construction is used in the integer-valued setting of Section~\ref{sec:main}.  See Biggs~\cite{Biggs1974} for the graph-theoretic background and Thurston~\cite{Thurston1990} for the height-function analogue.

\emph{Forward direction (potential $\Rightarrow$ cycle closure).}
Let $v_0\to v_1\to\cdots\to v_k=v_0$ be any closed walk. Then
\[
  \sum_{i=0}^{k-1}\Delta(v_i\to v_{i+1})
  \;=\;
  \sum_{i=0}^{k-1}\bigl(h(v_{i+1})-h(v_i)\bigr)
  \;=\;
  h(v_k)-h(v_0)
  \;=\;0.
\]

\emph{Converse (cycle closure $\Rightarrow$ potential).}
Fix a reference vertex $v_0$.
For $v\in V$, choose any path $\gamma_{v_0\to v}$ from $v_0$ to $v$ and define
\[
  h(v)\;:=\;\sum_{e\in\gamma_{v_0\to v}}\Delta(e).
\]
If $\gamma$ and $\gamma'$ are two such paths, then $\gamma\cdot(\gamma')^{-1}$
is a closed walk, so its $\Delta$-sum is $0$ by cycle closure, hence the two
path sums agree. Therefore $h$ is well-defined and satisfies
$\Delta(u\to v)=h(v)-h(u)$ on every directed edge.

\emph{Integer-valued case.}  If $\Delta(\vec{E})\subseteq\ZZ$ and $h(v_0)=0$, then every path sum is a finite sum of integers, so $h(v)\in\ZZ$ for all $v\in V$.
\end{proof}

\begin{remark}[Cohomological interpretation]
\label{rem:cohomological-meaning}
Theorem~\ref{thm:poincare} is equivalent to the vanishing of the first cohomology group $H^1(G;\RR)$ for any connected graph $G$.
Concretely: the group of $1$-cochains $C^1(G;\RR)$ decomposes as $\ker\delta^1 = \mathrm{im}\,\delta^0$ (cocycles = coboundaries), where $\delta^0 h(u\to v) := h(v)-h(u)$ and $\delta^1\Delta(f) := \sum_{e\in\partial f}\Delta(e)$ (face boundary sum).
In the integer-valued setting $\Delta\colon\vec{E}\to\ZZ$, the relevant group is $H^1(G;\ZZ)$, which is trivial for trees and---by the universal coefficient theorem---equals $\mathrm{Hom}(H_1(G;\ZZ),\ZZ)$ for general graphs.
For the tiling adjacency graph $G_{\mathcal{T}}$, the vanishing of $H^1$ is equivalent to the matching rules (Section~\ref{sec:main}), and the integer class $[\Delta_{\mathbf{n}}]\in H^1(G_{\mathcal{T}};\ZZ)$ measuring the obstruction to global bar continuity is the subject of Section~\ref{sec:cohomology}.
\end{remark}

Section~\ref{sec:main} applies Theorem~\ref{thm:poincare} to the Penrose setting: the abstract $1$-cochain $\Delta$ is replaced by the signed bar-crossing count $\Delta_{\mathbf{n}}$, and the abstract scalar potential $h$ becomes the Ammann height function.
The equivalence ``cycle closure $\Leftrightarrow$ matching rules'' is the central result of Section~\ref{sec:main}.

\section{The Tiling--Potential Correspondence}
\label{sec:main}

This section is the technical core of the paper.
Using the Penrose background from Section~\ref{sec:background} and the abstract algebraic tools from Section~\ref{sec:discrete-potential}, we prove that the signed Ammann bar-crossing count is a $1$-cochain on the tiling adjacency graph $G_{\mathcal{T}}$, and that the discrete Poincar\'{e} lemma (Theorem~\ref{thm:poincare}) applies to yield the classical Ammann height function as the resulting scalar potential.
The central result is Theorem~\ref{thm:matching-conservation}: for any \emph{candidate} decorated Penrose tiling, the four conditions---matching rules, bar continuity, cycle closure, and height-function existence---are mutually equivalent.
This equivalence is the precise algebraic sense in which matching rules are conservation laws (cocycle conditions).

The roadmap for this section goes as follows: \S\ref{sec:bar-postings} defines the tile-side
bar-crossing data and the gluing condition that turns it into a global
$1$-cochain; \S\ref{sec:patch-example} gives an explicit patch-level example and a
concrete matching-rule violation; \S\ref{sec:height-potential} identifies the
resulting potential with the Ammann height; \S\ref{sec:conservation} proves the
four-way equivalence for candidate tilings; \S\ref{sec:five-family} assembles all
five families and proves pentagrid reconstruction; \S\ref{sec:cohomology} explains
the cohomological meaning; and \S\ref{sec:algorithm} extracts a linear-time
validation procedure.

We now return to the Penrose tiling graph $G_{\mathcal{T}}$ of
\S\ref{sec:recog-graph} and show that the Ammann bar data of
\S\ref{sec:penrose-background} provide a natural antisymmetric edge
function on $G_{\mathcal{T}}$ (Definition~\ref{def:antisymmetric}).  The general theory of
\S\ref{sec:discrete-potential} then applies: cycle closure (Definition~\ref{def:cycle-closure})
yields a potential (Theorem~\ref{thm:poincare}, equation~\eqref{eq:potential}),
and the potential turns out to be the classical Ammann height function.

\subsection{Bar-crossing functions as antisymmetric 1-cochains}
\label{sec:bar-postings}

The construction of the bar-crossing cochain proceeds in two stages:
a \emph{tile-side} (half-edge) function defined from each tile's own
decoration alone, and a \emph{global} cochain that exists precisely when
adjacent tiles agree on every shared edge---i.e., when the matching rules
hold.
This two-stage approach makes the conservation-law interpretation exact and
removes any ambiguity in the candidate-tiling setting.

\begin{definition}[Tile-side bar-crossing function]
\label{def:half-edge-crossing}
Fix one of the five Ammann bar families of $\mathcal{T}$, with unit normal
vector $\mathbf{n}$.
For each tile $t \in \mathcal{T}$ and each directed edge $(u \to v)$ on the
boundary $\partial t$, define
\begin{equation}
  \Delta_{\mathbf{n}}^{(t)}(u \to v) :=
  \begin{cases}
    +1 & \text{if tile $t$'s local bar segment crosses $(uv)$
         and $\mathbf{n}\cdot(v-u) > 0$,} \\
    -1 & \text{if tile $t$'s local bar segment crosses $(uv)$
         and $\mathbf{n}\cdot(v-u) < 0$,} \\
     0 & \text{if no local bar segment of $t$ crosses $(uv)$.}
  \end{cases}
\label{eq:half-edge-crossing}
\end{equation}
Here ``local bar segment of $t$'' refers to the finite bar segment inside
tile $t$ determined solely by $t$'s own Penrose decoration
(Definition~\ref{def:local-global-bars}), with no reference to adjacent
tiles.
This function is well-defined for every candidate tiling:
each tile edge crosses at most one local bar segment per
family~\cite{BaakeGrimm2013}, so $\Delta_{\mathbf{n}}^{(t)}(u \to v)
\in \{-1,0,+1\}$ (see Figure~\ref{fig:sign-convention} for
the sign convention); and reversing the edge reverses the sign, so the
function is antisymmetric on $\partial t$.
\end{definition}

A concrete computation of these tile-side values appears in Example~\ref{ex:dart-patch} in \S\ref{sec:patch-example}.

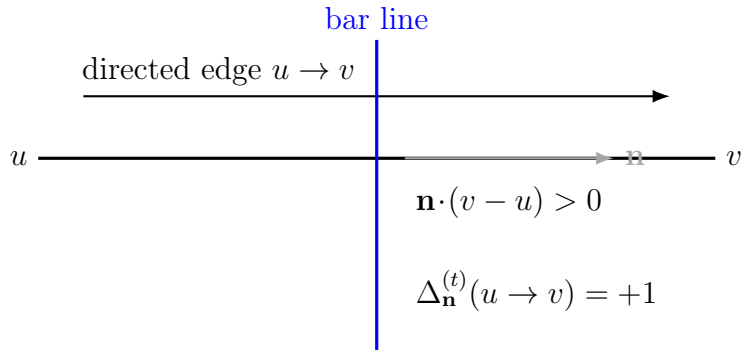
\begin{figure}[H]
\centering
\begin{tikzpicture}[scale=1.5,>=Latex]
  \coordinate (u) at (0,0);
  \coordinate (v) at (6,0);
  \draw[very thick] (u) -- (v);
  \node[left]  at (u) {$u$};
  \node[right] at (v) {$v$};
  \draw[->,thick] (0.4,0.55) -- (5.6,0.55);
  \node[above] at (1.6,0.55) {directed edge $u\to v$};
  \draw[very thick,blue] (3,-1.7) -- (3,1.05);
  \node[blue,anchor=south] at (3,1.05) {bar line};
  \draw[->,thick,gray!70] (3.25,0) -- (5.1,0)
        node[anchor=west] {$\mathbf{n}$};
  \node[anchor=north west] at (3.25,-0.15)
        {$\mathbf{n}\!\cdot\!(v-u)>0$};
  \node[anchor=north west] at (3.25,-0.90)
        {$\Delta_{\mathbf{n}}^{(t)}(u\to v)=+1$};
\end{tikzpicture}
\caption{Sign convention for Definition~\ref{def:half-edge-crossing}:
a crossing with $\mathbf{n}\!\cdot\!(v-u)>0$ contributes $+1$;
the reversed directed edge contributes $-1$.}
\label{fig:sign-convention}
\end{figure}

\begin{remark}[Why edge type alone does not suffice]
\label{rem:edge-type-insufficiency}
One might attempt to define $\Delta_{\mathbf{n}}^{(t)}$ purely from edge type
(``$+1$ for long edges, $0$ for short edges'').
This fails: such a rule assigns $+1$ regardless of traversal direction,
contradicting antisymmetry.
The bar-family normal $\mathbf{n}$ resolves this by providing an orientation
extrinsic to the edge, and different bar families yield different tile-side
functions $\Delta_{\mathbf{n}}^{(t)}$, each producing one of the five distinct Ammann height functions (whose independence is established by Theorem~\ref{thm:five-family}(ii) below).
\end{remark}

\begin{lemma}[Tile boundary sums vanish]
\label{lem:tile-sum-zero}
For every tile $t$ and every bar family~$\mathbf{n}$,
\begin{equation}
  \sum_{(u \to v) \in \partial t} \Delta_{\mathbf{n}}^{(t)}(u \to v) = 0,
\label{eq:tile-sum-zero}
\end{equation}
where $\partial t$ is traversed in the counterclockwise boundary orientation.
\end{lemma}

\begin{proof}
Each local bar segment inside $t$ is a straight segment parallel to the bar
direction~\cite{GrunbaumShephard1987,BaakeGrimm2013}.
For Penrose P2 tiles (darts and kites), the five Ammann bar families are
oriented so that their bars are never parallel to any tile edge and never
pass through any tile vertex~\cite{BaakeGrimm2013}; consequently, each local
bar segment crosses $\partial t$ transversally in the relative interior of
exactly two boundary edges---entering through one and exiting through another.

To see that the two crossings contribute opposite signs, we use the following
topological observation: the boundary $\partial t$ is a simple closed curve in
$\RR^2$, and the \emph{algebraic intersection number} of any closed curve with
any oriented line is zero.
Concretely, the bar of family~$\mathbf{n}$ is an oriented line with normal
$\mathbf{n}$.  Define $H^+ := \{\mathbf{n} \cdot x > c\}$ and
$H^- := \{\mathbf{n} \cdot x < c\}$ for the appropriate constant $c$.
As $\partial t$ is traversed counterclockwise, it starts in one open
half-plane, say $H^-$, crosses into $H^+$ at the entering edge
(contributing $\Delta^{(t)}_{\mathbf{n}} = +1$ since $\mathbf{n}\cdot(v-u)>0$
at that crossing), and must eventually cross back into $H^-$ at the exiting
edge (contributing $\Delta^{(t)}_{\mathbf{n}} = -1$ since
$\mathbf{n}\cdot(v'-u')<0$ at that crossing) in order to return to the
starting vertex.
Since $\partial t$ is a closed curve, the signed count of crossings with
the bar line equals zero.
Any bar segment not crossing $\partial t$ contributes only zeros.
Summing over all bar segments inside~$t$ gives zero.
\end{proof}

Before stating the gluing condition, we record the precise link between
edge markings and crossing values that makes the equivalence
(gluing $\Leftrightarrow$ matching) available.

\begin{remark}[From edge decorations to signed crossings]
\label{rem:decoration-to-crossing}
The $\{-1,0,+1\}$ crossing values are decoration-determined
(Lemma~\ref{lem:local-bars-determined}): the tile-side value
$\Delta^{(t)}_{\mathbf{n}}(u\to v)$ reads directly from the marking
that tile~$t$ places on $\{u,v\}$.
For the unique relevant family $\mathbf{n}^*$ associated with the
edge-direction class of $\{u,v\}$ (Remark~\ref{rem:transversality}(iii)),
equality of tile-side values is therefore equivalent to agreement of
decorations; see~\cite[Ch.~6]{BaakeGrimm2013} for the explicit
marking-to-crossing lookup.
\end{remark}

\begin{lemma}[Gluing condition and global cochain]
\label{lem:gluing}
Let $\mathcal{T}$ be a candidate decorated Penrose tiling (see Figure~\ref{fig:gluing-schematic} for a schematic).
In the Penrose P2 tiling, the five bar families correspond to five distinct edge-direction classes (Remark~\ref{rem:transversality}(iii)): each interior edge $\{u,v\}$ is crossed by bars of \emph{exactly one} bar family (the family whose normal $\mathbf{n}^*$ is parallel to the direction of $\{u,v\}$); for the other four families, both tile-side values are~$0$ trivially.
Fix a bar family with normal $\mathbf{n}$.
\begin{enumerate}[label=(\roman*),noitemsep]
\item \textbf{(Gluing equals matching.)}
For each interior edge $\{u,v\}$ shared by tiles $A$ and $B$:
\begin{equation}
  \Delta_{\mathbf{n}}^{(A)}(u \to v) = \Delta_{\mathbf{n}}^{(B)}(u \to v).
\label{eq:gluing-condition}
\end{equation}
Moreover, \eqref{eq:gluing-condition} holds if and only if the Penrose matching
rule holds on $\{u,v\}$ with respect to family~$\mathbf{n}$.
Here ``matching rule with respect to family $\mathbf{n}$'' means: the
markings on $\{u,v\}$ from $A$ and $B$ agree \emph{as seen by family
$\mathbf{n}$}.
If $\mathbf{n}$ is \emph{not} the family whose bars cross edges of type
$\{u,v\}$, both sides equal~$0$ trivially and the condition carries no
information.
The \emph{full} Penrose matching rule on $\{u,v\}$ holds if and only if
the gluing condition~\eqref{eq:gluing-condition} holds for the unique
family $\mathbf{n}^*$ that crosses $\{u,v\}$.
\item \textbf{(Global cochain.)}
$\Delta_{\mathbf{n}}$ on $G_{\mathcal{T}}$ exists if and only if the gluing
condition~\eqref{eq:gluing-condition} holds on every interior edge.
It is an antisymmetric edge function
\[
  \Delta_{\mathbf{n}}\colon \vec{E} \to \{-1,0,+1\}.
\]
When it exists, $\Delta_{\mathbf{n}}(u \to v)$ is the common value of
$\Delta_{\mathbf{n}}^{(A)}$ and $\Delta_{\mathbf{n}}^{(B)}$ on the shared edge.
\end{enumerate}
\end{lemma}

\begin{proof}
\emph{Part~(i).}
In the Penrose P2 decoration, each edge of each prototile carries a specific
arrow/arc marking, and whether a local bar segment of tile $t$ crosses a
given boundary edge $\{u,v\}$ of $t$---and in which direction relative to
$\mathbf{n}$---is determined entirely by the edge marking that $t$ places on
$\{u,v\}$ (Lemma~\ref{lem:local-bars-determined}); see also
Baake--Grimm~\cite[Ch.~6]{BaakeGrimm2013}.
Consequently, $\Delta_{\mathbf{n}}^{(t)}(u \to v)$ depends only on $t$'s
own marking of $\{u,v\}$.

\emph{($\Leftarrow$, matching implies gluing.)}
The Penrose matching rule on $\{u,v\}$ asserts that $A$ and $B$ place
identical markings on $\{u,v\}$.  Since $\Delta^{(t)}_{\mathbf{n}^*}(u\to v)$
depends only on $t$'s own marking, identical markings give identical
tile-side values.

\emph{($\Rightarrow$, gluing implies matching.)}
Each edge of direction class $\alpha$ carries an arrow decoration pointing in
one of exactly two directions, so $|M_\alpha|=2$.
By Remark~\ref{rem:transversality}(iv), the relevant bar family
$\mathbf{n}^*_\alpha$ crosses every edge of class $\alpha$, so the
tile-side crossing value is always $\pm 1$ (never $0$) for both markings.
The two markings correspond to opposite arrow orientations; swapping the
arrow direction negates the crossing value, yielding $+1$ for one marking
and $-1$ for the other.
Since $+1\neq-1$, the map $m\mapsto\Delta^{(t)}_{\mathbf{n}^*_\alpha}(u\to v)$
is injective on $M_\alpha$.
Hence equal tile-side values
$\Delta^{(A)}_{\mathbf{n}^*}(u\to v) = \Delta^{(B)}_{\mathbf{n}^*}(u\to v)$
imply equal markings, which is precisely the matching rule on $\{u,v\}$.

For any other family $\mathbf{n} \neq \mathbf{n}^*$, both sides equal~$0$
regardless of the matching rule, so the gluing condition holds trivially.

\emph{Part~(ii).}
If the gluing condition holds on every interior edge, define
$\Delta_{\mathbf{n}}(u \to v)$ as the common value for shared edges, and
$\Delta_{\mathbf{n}}^{(t)}(u \to v)$ for boundary edges (belonging to one
tile only).
Antisymmetry of the global function follows from antisymmetry within each
tile.
Conversely, if $\Delta_{\mathbf{n}}$ is a well-defined global antisymmetric
function, its value on each shared edge must equal both
$\Delta_{\mathbf{n}}^{(A)}$ and $\Delta_{\mathbf{n}}^{(B)}$, so the gluing
condition holds.
\end{proof}

\begin{figure}[H]
\centering
\begin{tikzpicture}[scale=1.5,>=Latex]
  \coordinate (u) at (0,0);
  \coordinate (v) at (4,0);
  \coordinate (a1) at (-0.6,1.3);
  \coordinate (a2) at (2.0,1.8);
  \coordinate (a3) at (4.6,1.3);
  \coordinate (b1) at (-0.6,-1.3);
  \coordinate (b2) at (2.0,-1.8);
  \coordinate (b3) at (4.6,-1.3);

  \draw[thick] (a1)--(u)--(v)--(a3)--(a2)--cycle;
  \draw[thick] (b1)--(u)--(v)--(b3)--(b2)--cycle;
  \draw[very thick] (u)--(v);

  \node[above] at (0.9, 0.7)  {$A$};
  \node[below] at (0.9,-0.7) {$B$};

  \draw[->,thick] (0.3,0.28)--(3.7,0.28);
  \node[above] at (0.85,0.28) {$u\to v$};
  \draw[->,thick] (3.7,-0.28)--(0.3,-0.28);
  \node[below] at (0.85,-0.28) {$v\to u$};

  \draw[->,thick,gray!70] (2.2,0) -- (3.5,0) node[anchor=north west] {$\mathbf{n}$};

  \draw[dashed,gray!70] (2,-1.6) -- (2,1.6);

  \node[left]  at (u) {$u$};
  \node[right] at (v) {$v$};
\end{tikzpicture}
\caption{Schematic of the gluing condition on a shared edge.  Tiles $A$ and $B$ induce tile-side values $\Delta^{(A)}_{\mathbf{n}}$ and $\Delta^{(B)}_{\mathbf{n}}$ on the same directed edge $u\to v$. The Penrose matching rule on $\{u,v\}$ (for the unique relevant family) is equivalent to equality of these induced values (Lemma~\ref{lem:gluing}).}
\label{fig:gluing-schematic}
\end{figure}
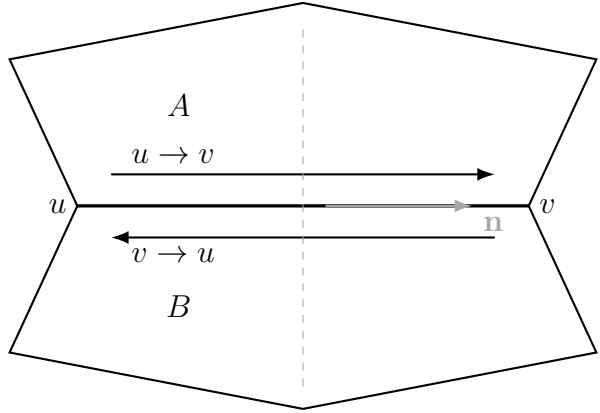

\begin{definition}[Bar-crossing function (global)]
\label{def:bar-crossing}
Let $\mathcal{T}$ satisfy the matching rules, so the gluing condition of
Lemma~\ref{lem:gluing} holds for every bar family.
The \emph{bar-crossing function} for family $\mathbf{n}$ is the globally
consistent antisymmetric edge function
\[
  \Delta_{\mathbf{n}}\colon \vec{E} \longrightarrow \{-1,0,+1\}
\]
given by $\Delta_{\mathbf{n}}(u \to v) = \Delta_{\mathbf{n}}^{(t)}(u \to v)$
for any tile $t$ incident to edge $\{u,v\}$.
By Lemma~\ref{lem:gluing}(ii), this value is independent of the choice of $t$.
\end{definition}

\begin{remark}[Conservation-law reading of the gluing condition]
\label{rem:gluing-conservation}
Each tile contributes a tile-side function $\Delta_{\mathbf{n}}^{(t)}$
balanced around its own boundary (Lemma~\ref{lem:tile-sum-zero}).
The matching rule is exactly the \emph{gluing condition}: adjacent tiles
must agree on the value they assign to every shared directed edge.
A mismatch---one tile assigns $+1$ and the other $-1$ or $0$ to the same
directed edge---is a local violation that prevents a globally consistent
cochain from existing (see Remark~\ref{rem:conservation} for the
detailed conservation-law interpretation).
\end{remark}

\subsection{A concrete patch example}
\label{sec:patch-example}

We ground the abstract machinery in explicit tile-face computations,
then show how a matching-rule violation destroys cycle closure.

\begin{figure}[H]
\centering
\begin{subfigure}[t]{0.48\textwidth}
\centering
\begin{tikzpicture}[scale=1.3,>=Latex]
  \coordinate (u) at (0,0);
  \coordinate (v) at (3,0);
  \coordinate (d1) at (-0.8,1.2);
  \coordinate (d2) at (1.5,1.6);
  \coordinate (d3) at (3.8,1.2);
  \coordinate (k1) at (-0.8,-1.2);
  \coordinate (k2) at (1.5,-1.6);
  \coordinate (k3) at (3.8,-1.2);

  \draw[thick] (d1)--(u)--(v)--(d3)--(d2)--cycle;
  \draw[thick] (k1)--(u)--(v)--(k3)--(k2)--cycle;
  \draw[very thick] (u)--(v);

  \draw[->,thick] (0.3,0.35)--(2.7,0.35);
  \node[above] at (0.7,0.35) {$u\to v$};

  \draw[->,thick,gray!70] (1.7,0.0) -- (2.7,0.0) node[anchor=north west] {$\mathbf{n}$};

  \draw[very thick,blue] (1.5,1.25)--(1.5,-1.25);
  \node[blue,anchor=north west] at (1.65,1.15) {bar};

  \node[above] at (0.65,0.65) {$D$};
  \node[below] at (0.65,-0.65) {$K$};
  \node[left]  at (u) {$u$};
  \node[right] at (v) {$v$};
\end{tikzpicture}
\caption{Matching: the bar continues across $\{u,v\}$.}
\end{subfigure}\hfill
\begin{subfigure}[t]{0.48\textwidth}
\centering
\begin{tikzpicture}[scale=1.3,>=Latex]
  \coordinate (u) at (0,0);
  \coordinate (v) at (3,0);
  \coordinate (d1) at (-0.8,1.2);
  \coordinate (d2) at (1.5,1.6);
  \coordinate (d3) at (3.8,1.2);
  \coordinate (k1) at (-0.8,-1.2);
  \coordinate (k2) at (1.5,-1.6);
  \coordinate (k3) at (3.8,-1.2);

  \draw[thick] (d1)--(u)--(v)--(d3)--(d2)--cycle;
  \draw[thick] (k1)--(u)--(v)--(k3)--(k2)--cycle;
  \draw[very thick] (u)--(v);

  \draw[->,thick] (0.3,0.35)--(2.7,0.35);
  \node[above] at (0.7,0.35) {$u\to v$};

  \draw[->,thick,gray!70] (1.7,0.0) -- (2.7,0.0) node[anchor=north west] {$\mathbf{n}$};

  \draw[very thick,blue] (1.5, 1.25)--(1.5, 0.07);
  \draw[very thick,blue] (1.65,-0.07)--(1.65,-1.25);
  \node[blue,anchor=north west] at (1.7,1.15) {broken bar};

  \node[above] at (0.65,0.65) {$D$};
  \node[below] at (0.65,-0.65) {$K'$};
  \node[left]  at (u) {$u$};
  \node[right] at (v) {$v$};
\end{tikzpicture}
\caption{Mismatch: the bar breaks at $\{u,v\}$.}
\end{subfigure}
\caption{Schematic version of Examples~\ref{ex:dart-patch}--\ref{ex:violation}.
The blue segment represents one local Ammann bar segment from the chosen family.
In the matching case the local segments glue to a single straight bar; in the
mismatched case they fail to glue, so no global cochain (and hence no global
potential) exists for that family.}
\label{fig:dart-kite-schematic}
\end{figure}
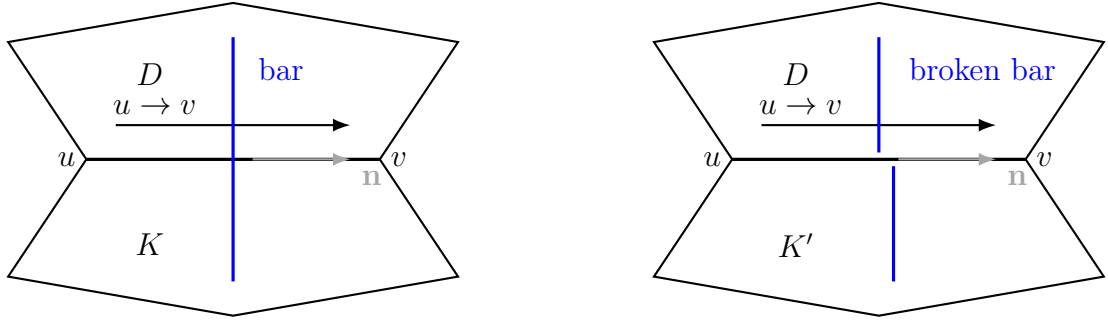

\begin{example}[Bar-crossing function on a dart--kite pair]
\label{ex:dart-patch}
(See Figure~\ref{fig:dart-kite-schematic}, left panel.)
Fix the bar family with unit normal $\mathbf{n}$.
Consider a dart tile $D$ with vertices $v_0, v_1, v_2, v_3$
(counterclockwise), where $\overline{v_0 v_1}$ and $\overline{v_3 v_2}$
are the two long edges and $\overline{v_1 v_2}$, $\overline{v_0 v_3}$
are the two short edges.
Suppose the bar family is oriented so that bars cross the two long edges.
The tile-side bar-crossing values $\Delta_{\mathbf{n}}^{(D)}$ on the dart
boundary are:
\begin{align*}
  \Delta_{\mathbf{n}}^{(D)}(v_0 \to v_1) &= +1
    &\text{(long edge $\overline{v_0v_1}$, crossed in $+\mathbf{n}$ direction)},\\
  \Delta_{\mathbf{n}}^{(D)}(v_1 \to v_2) &= 0
    &\text{(short edge $\overline{v_1v_2}$, no crossing)},\\
  \Delta_{\mathbf{n}}^{(D)}(v_2 \to v_3) &= -1
    &\text{(long edge $\overline{v_3v_2}$, crossed in $-\mathbf{n}$ direction)},\\
  \Delta_{\mathbf{n}}^{(D)}(v_3 \to v_0) &= 0
    &\text{(short edge $\overline{v_0v_3}$, no crossing)}.
\end{align*}
Tile boundary sum on $D$:\enskip $(+1) + 0 + (-1) + 0 = 0$.\enskip$\checkmark$
(Lemma~\ref{lem:tile-sum-zero}.)

Normalising $h_{\mathbf{n}}(v_0)=0$, the height values are
$h_{\mathbf{n}}(v_1)=1$, $h_{\mathbf{n}}(v_2)=1$, $h_{\mathbf{n}}(v_3)=0$.

Now adjoin a kite $K$ sharing the long edge $\overline{v_0 v_1}$ with $D$.
The Penrose matching rule forces the kite's decoration to agree with the
dart's on this shared edge; by Lemma~\ref{lem:gluing}(i), this is
equivalent to the gluing condition
$\Delta_{\mathbf{n}}^{(K)}(v_0 \to v_1) = \Delta_{\mathbf{n}}^{(D)}(v_0 \to v_1) = +1$.
The global bar-crossing function $\Delta_{\mathbf{n}}$
(Definition~\ref{def:bar-crossing}) is therefore well-defined on
$\overline{v_0v_1}$, and the Ammann bar continues straight from $D$ into $K$.
The kite boundary sum is also $0$ (Lemma~\ref{lem:tile-sum-zero}), and the
height values $h_{\mathbf{n}}(v_0)=0$, $h_{\mathbf{n}}(v_1)=1$ extend
consistently to the kite's remaining vertices.
(See Appendix~\ref{app:tile-calcs} for a complete enumeration of tile
orientations.)
\end{example}

\begin{example}[A matching-rule violation breaks cycle closure]
\label{ex:violation}
(See Figure~\ref{fig:dart-kite-schematic}, right panel.)
Keep the dart $D$ from Example~\ref{ex:dart-patch} but adjoin a
\emph{mismatched} kite $K'$ that violates the matching rule on
$\overline{v_0 v_1}$: the decoration on $K'$'s side of this edge disagrees
with the dart's.
By Lemma~\ref{lem:gluing}(i), this decoration mismatch is exactly the
failure of the gluing condition:
\[
  \Delta_{\mathbf{n}}^{(D)}(v_0 \to v_1) = +1
  \;\neq\;
  \Delta_{\mathbf{n}}^{(K')}(v_0 \to v_1).
\]
If $K'$'s mismatched bar segment crosses $\overline{v_0 v_1}$ in the
opposite direction then $\Delta_{\mathbf{n}}^{(K')}(v_0 \to v_1) = -1$;
if $K'$ places no bar on this edge then
$\Delta_{\mathbf{n}}^{(K')}(v_0 \to v_1) = 0$.
In either case, by Lemma~\ref{lem:gluing}(ii), no globally consistent
$\Delta_{\mathbf{n}}$ exists on $G_{\mathcal{T}}$.
Because the decorations disagree, the local bar segment inside $D$ and the
segment inside $K'$ do not join at $\overline{v_0 v_1}$; the bar
\emph{terminates} at the shared edge (Remark~\ref{rem:conservation}).
The gluing-condition check (Remark~\ref{rem:algorithm}) detects this
directly as a gluing violation on edge $\{v_0,v_1\}$ for bar family~$\mathbf{n}$.
\end{example}

\subsection{Height function as discrete potential}
\label{sec:height-potential}

The following proposition assembles bar continuity, the discrete
Poincar\'{e} lemma, and the classical Ammann height function into a
single statement: the potential produced by the lemma and the height
function are the same object, constructed by the same path integral.

\begin{proposition}[Height function as discrete potential]
\label{prop:height-potential}
Let $\mathcal{T}$ be a Penrose tiling satisfying the matching rules, and let
$\Delta_{\mathbf{n}}$ be the bar-crossing function of
Definition~\ref{def:bar-crossing}.
Then:
\begin{enumerate}[label=(\roman*)]
\item \textbf{(Cycle closure.)}
For every closed walk $\gamma = (v_0 \to v_1 \to \cdots \to v_k = v_0)$
in $G_{\mathcal{T}}$,
\begin{equation}
  \sum_{i=0}^{k-1} \Delta_{\mathbf{n}}(v_i \to v_{i+1}) = 0.
\label{eq:main-cycle}
\end{equation}

\item \textbf{(Potential existence.)}
By the discrete Poincar\'{e} lemma (Theorem~\ref{thm:poincare}),
there exists a potential
\[
  h_{\mathbf{n}}\colon V \longrightarrow \ZZ
\]
such that
$\Delta_{\mathbf{n}}(u \to v) = h_{\mathbf{n}}(v) - h_{\mathbf{n}}(u)$
for all $(u \to v) \in \vec{E}$, unique up to a global additive constant.

\item \textbf{(Identification.)}
The potential $h_{\mathbf{n}}$ coincides with the classical Ammann height
function of $\mathcal{T}$ for the bar family with normal~$\mathbf{n}$.
This is a definitional identification: both are constructed as path
integrals of the signed bar-crossing count from a common reference vertex.
\end{enumerate}
\end{proposition}

\begin{proof}
\emph{Part~(i): Cycle closure from bar continuity.}

Since $\mathcal{T}$ satisfies the matching rules, each Ammann bar forms a
continuous straight line extending across the entire
tiling~\cite{BaakeGrimm2013}.
We prove that $\Delta_{\mathbf{n}}$ has zero sum around any closed walk.

Let $\gamma$ be a closed walk and let $\ell$ be any single Ammann bar from the
chosen family.  The bar~$\ell$ is a straight line dividing the plane into two
open half-planes, $H^+$ (the $\mathbf{n}$-side) and $H^-$ (the
$-\mathbf{n}$-side).
Since $\gamma$ is a closed walk, it starts and ends at the same vertex.
Each time $\gamma$ crosses $\ell$ from $H^-$ to $H^+$, it must eventually
cross back from $H^+$ to $H^-$ (and vice versa) in order to return to its
starting point.
Therefore the number of positive crossings of $\ell$ equals the number of
negative crossings.
The contribution of bar $\ell$ to
$\sum_i \Delta_{\mathbf{n}}(v_i \to v_{i+1})$ is zero.

Since $\gamma$ has finitely many edges and each edge crosses at most one bar,
only finitely many bars contribute.  The total sum decomposes as a sum over
individual bars, each contributing zero:
\[
  \sum_{i=0}^{k-1} \Delta_{\mathbf{n}}(v_i \to v_{i+1})
  = \sum_{\ell \in \mathcal{B}_{\mathbf{n}}(\mathcal{T})}
    \bigl(\text{signed crossings of $\ell$ by $\gamma$}\bigr)
  = 0.
\]

\emph{Part~(ii): Potential existence.}
The function $\Delta_{\mathbf{n}}$ is antisymmetric (Definition~\ref{def:bar-crossing})
and satisfies cycle closure (Part~(i)).
The graph $G_{\mathcal{T}}$ is connected (Remark~\ref{rem:graph-properties}).
The discrete Poincar\'{e} lemma (Theorem~\ref{thm:poincare}) yields the
potential $h_{\mathbf{n}}$.
Since $\Delta_{\mathbf{n}}$ takes values in $\{-1,0,+1\}$, the potential
takes values in $\ZZ$ (after normalizing $h_{\mathbf{n}}(v_0) = 0$).

\emph{Part~(iii): Identification with the Ammann height.}
The classical Ammann height function $h^{\mathrm{cl}}_{\mathbf{n}}$ is
defined by assigning to each vertex $v$ the signed number of bars of family
$\mathbf{n}$ crossed (with sign $+1$ if crossed in the $\mathbf{n}$-direction,
$-1$ if crossed in the $-\mathbf{n}$-direction) along \emph{any} path from a
reference vertex $v_0$ to $v$~\cite{GrunbaumShephard1987,BaakeGrimm2013}.
Path-independence of this count for a valid Penrose tiling is precisely what
Part~(i) (cycle closure) and Part~(ii) (Theorem~\ref{thm:poincare}) establish:
the path integral of $\Delta_{\mathbf{n}}$ is independent of the chosen path.

It remains to check that ``path integral of $\Delta_{\mathbf{n}}$'' and
``signed bar count along the path'' are the same thing.
By definition of $\Delta_{\mathbf{n}}$ (Definition~\ref{def:bar-crossing}
and Definition~\ref{def:half-edge-crossing}), each directed edge $(u\to v)$
contributes $+1$, $-1$, or $0$ according to whether a bar of family $\mathbf{n}$
is crossed in the $\mathbf{n}$-direction, the $-\mathbf{n}$-direction, or not at
all.  Hence $\sum_{e \in \gamma} \Delta_{\mathbf{n}}(e)$ is exactly the signed
bar count along $\gamma$.

Normalising $h_{\mathbf{n}}(v_0) = 0 = h^{\mathrm{cl}}_{\mathbf{n}}(v_0)$, we
conclude $h_{\mathbf{n}} = h^{\mathrm{cl}}_{\mathbf{n}}$.
\end{proof}

\subsection{Matching rules as conservation laws}
\label{sec:conservation}

The following theorem is the main result of this paper: it shows that, for any candidate Penrose tiling, the matching rules, bar continuity, cycle closure, and height-function existence are all equivalent.
Throughout this subsection, ``conservation law'' refers to the discrete cocycle (closedness) condition for a $1$-cochain in CW cohomology; this is a purely algebraic condition and does not invoke any Noether-type symmetry principle.

Proposition~\ref{prop:height-potential} establishes the cochain--potential correspondence for a \emph{valid} tiling. The following theorem extends the picture to \emph{candidate} tilings by showing that matching rules, bar continuity, cycle closure, and height-function existence are all equivalent---without presupposing any of them.
The classical equivalence (i)$\Leftrightarrow$(ii) is recalled from \cite{GrunbaumShephard1987,BaakeGrimm2013}; the remaining implications follow from the half-edge/gluing construction of \S\ref{sec:bar-postings}.

\begin{theorem}[Matching rules as conservation laws]
\label{thm:matching-conservation}
Let $\mathcal{T}$ be a candidate decorated Penrose tiling in the sense of
Definition~\ref{def:candidate-tiling} (not necessarily satisfying the matching
rules).
The following conditions are equivalent:
\begin{enumerate}[label=(\roman*)]
\item \textbf{(Matching rules.)}
$\mathcal{T}$ satisfies the Penrose matching rules (edge decorations are
consistent across all shared edges).

\item \textbf{(Bar continuity.)}
For each of the five bar family directions, the Ammann bars of $\mathcal{T}$
form continuous straight lines (no bar terminates inside the tiling), i.e.,
every component of $\mathcal{B}_{\mathbf{n}}(\mathcal{T})$ is an entire line
for each bar-family normal~$\mathbf{n}$
(Definition~\ref{def:local-global-bars}).

\item \textbf{(Cycle closure.)}
For each bar family direction $\mathbf{n}$, the gluing condition
\eqref{eq:gluing-condition} holds on every interior edge---equivalently
(Lemma~\ref{lem:gluing}(ii)), the global bar-crossing cochain
$\Delta_{\mathbf{n}}\colon\vec{E}\to\{-1,0,+1\}$ exists---and
$\Delta_{\mathbf{n}}$ satisfies cycle closure on $G_{\mathcal{T}}$.

\item \textbf{(Height function.)}
For each bar family $\mathbf{n}$, a function $h_{\mathbf{n}}\colon V\to\ZZ$
exists such that
\begin{equation}
  h_{\mathbf{n}}(v) - h_{\mathbf{n}}(u) = \Delta_{\mathbf{n}}^{(t)}(u\to v)
\label{eq:height-consistency}
\end{equation}
for every directed edge $(u\to v)$ and every tile $t$ incident to $\{u,v\}$.
\end{enumerate}
\end{theorem}

\begin{proof}
\emph{(i)~$\Rightarrow$~(ii).}
This is a classical result: matching rules force Ammann bar segments
(determined locally by tile decorations) to join consistently across
shared edges, producing infinite straight
lines~\cite{GrunbaumShephard1987,BaakeGrimm2013}.

\emph{(ii)~$\Rightarrow$~(iii).}
Bar continuity means every bar of family $\mathbf{n}$ is an entire straight
line crossing the tiling.  In particular, whenever a bar crosses a shared
interior edge $\{u,v\}$, it enters tile $A$ through $\{u,v\}$ and exits tile
$B$ through the same edge (or vice versa), so both tiles assign the same
direction of crossing to $\{u,v\}$.
Hence $\Delta_{\mathbf{n}}^{(A)}(u\to v) = \Delta_{\mathbf{n}}^{(B)}(u\to v)$
for every interior edge: the gluing condition~\eqref{eq:gluing-condition} holds,
and by Lemma~\ref{lem:gluing}(ii) the global cochain $\Delta_{\mathbf{n}}$
exists.
Cycle closure then follows by the same argument as in
Proposition~\ref{prop:height-potential}(i): since each bar is an entire straight line, any
closed walk crosses it equally often in each direction, giving zero signed
crossing count.

\emph{(iii)~$\Rightarrow$~(iv).}
By assumption, $\Delta_{\mathbf{n}}$ exists as a global antisymmetric function
and satisfies cycle closure.
The graph $G_{\mathcal{T}}$ is connected (Remark~\ref{rem:graph-properties}).
By the discrete Poincar\'{e} lemma (Theorem~\ref{thm:poincare}), there exists
$h_{\mathbf{n}}\colon V \to \ZZ$ with
$\Delta_{\mathbf{n}}(u\to v) = h_{\mathbf{n}}(v) - h_{\mathbf{n}}(u)$.
Since $\Delta_{\mathbf{n}}(u\to v)$ is the common value of
$\Delta_{\mathbf{n}}^{(A)}(u\to v)$ and $\Delta_{\mathbf{n}}^{(B)}(u\to v)$
for any tiles $A,B$ incident to $\{u,v\}$ (Lemma~\ref{lem:gluing}(ii)),
condition~\eqref{eq:height-consistency} holds for every edge and every
incident tile.

\emph{(iv)~$\Rightarrow$~(i).}
Let $\{u,v\}$ be any interior edge shared by tiles $A$ and $B$.
By condition~\eqref{eq:height-consistency},
\[
  \Delta_{\mathbf{n}}^{(A)}(u\to v)
  \;=\; h_{\mathbf{n}}(v) - h_{\mathbf{n}}(u)
  \;=\; \Delta_{\mathbf{n}}^{(B)}(u\to v).
\]
Hence the gluing condition~\eqref{eq:gluing-condition} holds on $\{u,v\}$,
and by Lemma~\ref{lem:gluing}(i) the Penrose matching rule holds on $\{u,v\}$.
Since this applies to every interior edge and every bar family, (i) follows.
\end{proof}

\begin{remark}[The conservation-law interpretation]
\label{rem:conservation}
Each edge $\{u,v\}$ shared by tiles $A$ and $B$ is a comparison event. The antisymmetry condition requires $\Delta_\mathbf{n}(u\to v) = -\Delta_\mathbf{n}(v\to u)$: for a fixed orientation of the edge, the value at the directed edge $(u\to v)$ is the negative of the value at the reverse $(v\to u)$.
The gluing (matching) condition is \emph{separate}: it requires that $\Delta_\mathbf{n}^{(A)}(u\to v) = \Delta_\mathbf{n}^{(B)}(u\to v)$, i.e., both incident tiles assign the \emph{same} value to the directed edge $(u\to v)$. This is a two-sided consistency condition: the two incident tiles must assign
the same crossing value to the same directed edge; antisymmetry applies only to
reversing the direction.
A matching-rule violation at $\{u,v\}$ means $\Delta^{(A)} \neq
\Delta^{(B)}$, breaking the continuity of the Ammann bar and destroying
cycle closure and the global potential.
In this sense, matching rules are the tiling-theoretic expression of
conservation of the height potential.
\end{remark}

\subsection{Five-family validity criterion and pentagrid reconstruction}
\label{sec:five-family}

The preceding results treat one bar family at a time.
Assembling all five simultaneously yields a joint validity criterion and,
via identification of the height functions with de~Bruijn's strip
indices~\cite{deBruijn1981}, an explicit formula recovering every vertex
position from five integers.
Label the five Ammann bar families by
$\mathbf{n}_k = (\cos(2\pi k/5),\, \sin(2\pi k/5))$, $k = 0,1,2,3,4$.

\begin{theorem}[Five-family validity and pentagrid reconstruction]
\label{thm:five-family}
Let $\mathcal{T}$ be a candidate Penrose tiling, and fix a reference
vertex $v_0 \in V$ with $h_{\mathbf{n}_k}(v_0) = 0$ for all $k$.
\begin{enumerate}[label=(\roman*)]
\item \textbf{(Full validity criterion.)}
$\mathcal{T}$ satisfies the Penrose matching rules if and only if, for
every $k \in \{0,1,2,3,4\}$, a globally consistent height function
$h_{\mathbf{n}_k}\colon V \to \ZZ$ exists satisfying the
height-consistency equation~\eqref{eq:height-consistency} for every
directed edge $(u\to v)$ and every tile $t$ incident to $\{u,v\}$.

\item \textbf{(Pentagrid reconstruction.)}
If $\mathcal{T}$ is valid, then for every vertex $v \in V$:
\begin{equation}
  v - v_0
  \;=\;
  \frac{2}{5}
  \sum_{k=0}^{4} h_{\mathbf{n}_k}(v)\,\mathbf{n}_k.
\label{eq:pentagrid-reconstruction}
\end{equation}
In particular, the map
$\Phi\colon V \to \ZZ^5$,\; $\Phi(v) = (h_{\mathbf{n}_0}(v),\ldots,h_{\mathbf{n}_4}(v))$,
is \emph{injective}: distinct vertices carry distinct height-function
$5$-tuples.
The image $\Phi(V)$ lies in a $3$-dimensional affine subspace
$\mathcal{L} \subset \RR^5$, and within $\mathcal{L}$ the discrete set
$\Phi(V)\subset\ZZ^5$ forms the standard lifted stepped surface of the
Penrose cut-and-project construction.  The tiling $\mathcal{T}$ can be
reconstructed from $\Phi(V)$ via de~Bruijn's pentagrid
duality~\cite{deBruijn1981}.
\end{enumerate}
\end{theorem}

\begin{proof}
\emph{Part~(i).}
By Lemma~\ref{lem:gluing}(i), each interior edge $\{u,v\}$ has a unique
``relevant'' bar family $\mathbf{n}^*(u,v)$ (the one whose bars cross
edges of that geometric type); the matching rule on $\{u,v\}$ holds if
and only if the gluing condition for $\mathbf{n}^*(u,v)$ holds on
$\{u,v\}$.
Since the five bar families correspond to the five edge-direction classes
of P2 Penrose tilings, every interior edge belongs to one class, so every
edge is tested by its one relevant family.
Applying Theorem~\ref{thm:matching-conservation} (condition~(i)
$\Leftrightarrow$ condition~(iv)) to each family $\mathbf{n}_k$
separately, and taking the conjunction over $k = 0,\ldots,4$:
the full matching rule holds on every edge $\Leftrightarrow$ height
functions exist for all five families.

\emph{Part~(ii): strip-index identification.}
De~Bruijn's strip index $m_k^\mathrm{dB}(v)$ for vertex $v$ in family $k$
is the unique integer $m$ such that $v$ lies between the $m$th and
$(m+1)$th bar of family $k$, counted from the reference vertex $v_0$ in
the $\mathbf{n}_k$-direction, with the sign convention
\[
  m_k^\mathrm{dB}(v) > m_k^\mathrm{dB}(v_0)
  \;\text{ iff }\; \mathbf{n}_k\cdot(v - v_0) > 0.
\]
We claim $m_k^\mathrm{dB}(v) = h_{\mathbf{n}_k}(v)$ for all $v\in V$.

By Definition~\ref{def:half-edge-crossing}, for each directed edge
$(u\to v)$ the value $\Delta_{\mathbf{n}_k}(u\to v) = +1$ if a bar of
family $k$ is crossed with $\mathbf{n}_k\cdot(v-u)>0$ (the bar is crossed
in the increasing-$m$ direction), $-1$ if crossed with
$\mathbf{n}_k\cdot(v-u)<0$, and $0$ if no bar is crossed.
Hence the path integral $\sum_{e\in\gamma}\Delta_{\mathbf{n}_k}(e)$ from
$v_0$ to $v$ counts the net number of bars crossed in the
$+\mathbf{n}_k$-direction minus those crossed in the $-\mathbf{n}_k$-direction.
This net signed count equals $m_k^\mathrm{dB}(v) - m_k^\mathrm{dB}(v_0) =
m_k^\mathrm{dB}(v)$ (using $m_k^\mathrm{dB}(v_0)=0$).
Therefore $h_{\mathbf{n}_k}(v) = m_k^\mathrm{dB}(v)$ for every $v$.

\emph{Reconstruction formula.}
De~Bruijn's projection formula~\cite[eq.~(2.2)]{deBruijn1981} gives
\[
  v - v_0
  = \frac{2}{5}\sum_{k=0}^{4} m_k^\mathrm{dB}(v)\,\mathbf{n}_k
  = \frac{2}{5}\sum_{k=0}^{4} h_{\mathbf{n}_k}(v)\,\mathbf{n}_k,
\]
using the identity
$\sum_{k=0}^4 (\mathbf{n}_k \otimes \mathbf{n}_k) = \tfrac{5}{2} I_2$,
which holds because $\mathbf{n}_k = (\cos(2\pi k/5),\sin(2\pi k/5))$ are
five equally spaced unit vectors: the diagonal entry
$\sum_k\cos^2(2\pi k/5) = \tfrac{5}{2}$ (using $\cos^2\theta =
\tfrac{1+\cos 2\theta}{2}$ and $\sum_{k=0}^4\cos(4\pi k/5)=0$), and the
off-diagonal entry $\sum_k\cos(2\pi k/5)\sin(2\pi k/5) = \tfrac{1}{2}\sum_k
\sin(4\pi k/5) = 0$.

\emph{Injectivity.}
If $\Phi(u) = \Phi(v)$, then $u - v = \tfrac{2}{5}\sum_k
[h_{\mathbf{n}_k}(u)-h_{\mathbf{n}_k}(v)]\mathbf{n}_k = \mathbf{0}$
by~\eqref{eq:pentagrid-reconstruction}, so $u = v$.
Hence $\Phi$ is injective.

\emph{Image in a 3-dimensional affine subspace (and the lifted surface viewpoint).} The five strip indices are not independent: the pentagrid constraint (cf.~\cite[Sec.~2]{deBruijn1981}) is that $\sum_{k=0}^4 m_k^\mathrm{dB}(v)$ is constant on $V$ (the value depends only on the choice of reference vertex and the pentagrid offsets, not on $v$), and a second linear relation
\[
  \sum_{k=0}^{4}(-1)^k\,m_k^\mathrm{dB}(v) \;=\; \text{const}
\]
holds for all $v \in V$ (see~\cite[Sec.~2]{deBruijn1981}).
These two independent affine constraints cut out a $3$-dimensional affine subspace $\mathcal{L} \subset \RR^5$ containing $\Phi(V)$.
Within $\mathcal{L}$, the discrete set $\Phi(V)\subset\ZZ^5$ forms the standard \emph{lift} (a $2$-dimensional stepped surface) associated with the Penrose cut-and-project construction; see~\cite{deBruijn1981} for the explicit linear-algebraic description and its relation to the embedded physical plane in $\RR^5$.
\end{proof}

\subsection{Cohomological interpretation}
\label{sec:cohomology}

The cycle-closure condition has a natural cohomological reading that reinterprets the results of
\S\S\ref{sec:height-potential}--\ref{sec:five-family} in the language of cellular and pattern-equivariant cohomology. Let $X_{\mathcal{T}}$ denote the $2$-dimensional CW complex underlying the tiling: its $0$-cells are tiling vertices, its $1$-cells are tile edges, and its $2$-cells are the tile interiors attached along their boundary cycles.
Then $X_{\mathcal{T}}$ is homeomorphic to the plane and has $1$-skeleton $G_{\mathcal{T}}$.
The bar-crossing function $\Delta_{\mathbf{n}}$ can be viewed as a cellular $1$-cochain in $C^1(X_{\mathcal{T}};\ZZ)$ (equivalently, an antisymmetric function on directed edges). For a $2$-cell $f$ (a tile), the coboundary satisfies
\[
  (\delta \Delta_{\mathbf{n}})(f) \;=\; \sum_{e \in \partial f}\Delta_{\mathbf{n}}(e),
\]
so vanishing of the oriented boundary sums is the cocycle condition $\delta\Delta_{\mathbf{n}} = 0$.

For the planar complex $X_{\mathcal{T}}$, the cocycle condition $\delta\Delta_{\mathbf{n}}=0$ is equivalent to cycle closure on all closed walks, not merely simple cycles. The key step is the following: any closed walk $\gamma$ in $G_{\mathcal{T}}$ can be decomposed, in $C_1(X_{\mathcal{T}};\ZZ)$, as an integer linear combination of tile boundaries $\partial f_i$:
\begin{equation}
  \gamma \;=\; \sum_i \epsilon_i\,\partial f_i
  \quad\text{in } C_1(X_{\mathcal{T}};\ZZ),
\label{eq:cycle-decomp}
\end{equation}
where $\epsilon_i \in \{+1,-1\}$ are signs from the orientation, and the sum is finite. (This holds because $H_1(X_{\mathcal{T}};\ZZ)=0$---the plane $\mathbb{R}^2$ is simply connected---so every $1$-cycle is a $1$-boundary, i.e., a sum of
$2$-cell boundaries.)
Evaluating $\Delta_{\mathbf{n}}$ on both sides of~\eqref{eq:cycle-decomp}:
\[
  \sum_{e\in\gamma}\Delta_{\mathbf{n}}(e)
  = \sum_i \epsilon_i \sum_{e\in\partial f_i}\Delta_{\mathbf{n}}(e)
  = \sum_i \epsilon_i \cdot 0 = 0,
\]
using $\delta\Delta_{\mathbf{n}}=0$. Hence, face-sum vanishing implies cycle closure on all closed walks. The converse (cycle closure implies face-sum vanishing) is immediate since each tile boundary $\partial f$ is itself a closed walk.

The existence of a potential $h$ with $\Delta_{\mathbf{n}} = \delta_0 h$
is exactly the statement that $\Delta_{\mathbf{n}}$ is a coboundary, hence
$[\Delta_{\mathbf{n}}] = 0 \in H^1(X_{\mathcal{T}};\ZZ)$:
the $1$-cochain is \emph{exact}, not merely closed.

Moreover, $\Delta_{\mathbf{n}}$ is \emph{pattern-equivariant}: its value
on any edge depends only on a finite neighborhood of that edge
(the local tile type and orientation).
Pattern-equivariant cochains generate the \v{C}ech cohomology of the
tiling hull via the Kellendonk--Putnam isomorphism~\cite{KellendonkPutnam2006}.
For Penrose tilings, $\Hcech^1(\Omega_P;\ZZ) \cong \ZZ^5$, generated by
the five bar-crossing cocycles; see~\cite{Sadun2008}.
In the language of Forrest--Hunton--Kellendonk~\cite{ForrestHuntonKellendonk2002}, these five classes generate $\check{H}^1(\Omega_P;\ZZ)\cong\ZZ^5$.
Theorem~\ref{thm:five-family}(ii) adds an explicit consequence that is implicit in the FHK framework but not stated there: the five corresponding height functions jointly recover every vertex position via de~Bruijn's pentagrid formula~\eqref{eq:pentagrid-reconstruction}.
Each $\Delta_{\mathbf{n}}$ is exact as an ordinary 1-cocycle on the
contractible plane (Proposition~\ref{prop:height-potential}(ii)), but represents a
non-trivial class in pattern-equivariant cohomology: any potential $h_{\mathbf{n}}$ is necessarily non-local (not pattern-equivariant): $h_{\mathbf{n}}(v)$ is a path integral of $\Delta_{\mathbf{n}}$ from a reference vertex $v_0$, accumulating signed bar crossings over an arbitrarily long path; no finite neighbourhood of $v$ determines $h_{\mathbf{n}}(v)$. See~\cite{Sadun2008} or~\cite{ForrestHuntonKellendonk2002} for the formal statement.

The bar-crossing cochains thus exhibit a triad of properties---locally readable,
globally conserved, exact in ordinary cohomology but non-trivial in PE
cohomology---that will be formalised as conditions CF1--CF3 in
\S\ref{sec:conservation-forced}, where the resulting \emph{recognition gap}
is named and quantified (Remark~\ref{rem:recognition-gap-penrose}).

\subsection{Tiling validation}
\label{sec:algorithm}

\begin{proposition}[Linear-time tiling validation]
\label{prop:validation}
Let $\mathcal{T}$ be a finite simply-connected candidate decorated Penrose
tiling with $n$ vertices.
Then:
\begin{enumerate}[label=(\roman*),noitemsep]
\item The Penrose matching rules can be verified by checking, for each of the
five bar families, the gluing condition~\eqref{eq:gluing-condition} on every
interior edge.  By Euler's formula, the interior edge set has cardinality
$O(n)$, so the total cost of all five passes is $O(n)$.
\item The tiling satisfies the matching rules if and only if all five gluing
checks pass (Theorem~\ref{thm:five-family}(i)).
\item A gluing failure $\Delta_{\mathbf{n}}^{(A)}(u\to v)\neq\Delta_{\mathbf{n}}^{(B)}(u\to v)$
at edge $\{u,v\}$ identifies both the violating edge and the responsible bar
family, enabling targeted localisation of matching-rule violations.
\end{enumerate}
\end{proposition}

\begin{proof}
By Lemma~\ref{lem:gluing}(i), the gluing condition on a given edge $\{u,v\}$
for the unique relevant family $\mathbf{n}^*$ is equivalent to the Penrose
matching rule on $\{u,v\}$.
A planar graph on $n$ vertices has $O(n)$ edges by Euler's formula ($|E|\leq 3n-6$),
so five passes over the edge set cost $O(n)$ in total.
Correctness (part~(ii)) follows immediately from
Theorem~\ref{thm:five-family}(i); error localisation (part~(iii)) follows
from Lemma~\ref{lem:gluing}(ii).
\end{proof}

\begin{remark}[Linear-time validation with error localisation]
\label{rem:algorithm}
Theorem~\ref{thm:matching-conservation} yields an immediate validation procedure for a finite simply connected candidate patch $\mathcal{T}$ with $n$ vertices: for each of the five bar families, check the gluing condition~\eqref{eq:gluing-condition} on every interior edge.
By Lemma~\ref{lem:gluing}(i), each such check is equivalent to the matching rule on that edge for the unique relevant family. By Euler's formula, a planar graph with $n$ vertices has $O(n)$ edges,
so five passes suffice and the total cost is $O(n)$.

A naive decoration check (comparing arrow/arc markings on each shared edge) also runs in $O(n)$.  The point of the cochain formulation is not asymptotic speed but \emph{structured diagnosis}: a gluing failure $\Delta_{\mathbf{n}}^{(A)}(u\to v)\neq\Delta_{\mathbf{n}}^{(B)}(u\to v)$ identifies both the violating edge $\{u,v\}$ and the responsible bar family~$\mathbf{n}$, enabling targeted repair of the candidate tiling.

By Theorem~\ref{thm:five-family}(i), the tiling is valid if and only if all five checks pass. Since the five edge classes are disjoint, the five passes can be run in parallel; the total work remains $O(n)$.
\end{remark}

The bar-crossing construction of Section~\ref{sec:main} is developed in full detail for Penrose P2; the CPT framework of \S\ref{sec:cpt-background} already shows how it extends.
Section~\ref{sec:conservation-forced} asks whether the structural triad identified here---locally readable cochains (CF1--CF2), global conservation (CF1), and non-local potentials (CF3)---characterises a broader class of aperiodic tilings.
The answer is affirmative for all primitive-substitution canonical projection tilings, and the precise formulation is Definition~\ref{def:conservation-forced} and Theorem~\ref{thm:cpt-conservation-forced}.

\section{Conservation-Forced Aperiodic Order}
\label{sec:conservation-forced}

Section~\ref{sec:main} established the conservation structure for Penrose P2 in detail. This section asks: does that structure generalise? The answer is yes, for the entire class of primitive-substitution canonical projection tilings (CPTs). We formalise the essential properties as a five-condition definition (Definition~\ref{def:conservation-forced}), prove a single general theorem covering all CPTs (Theorem~\ref{thm:cpt-conservation-forced}), verify the four standard families explicitly, and identify the open classification problem that frames the entire section.

The results of \S\ref{sec:main} establish that the five Penrose bar-crossing cochains are (i)~locally computable (pattern-equivariant), (ii)~globally conserved (cycle closure holds), and (iii)~exact in ordinary cohomology but non-trivial in pattern-equivariant cohomology. This is the formal version of the conservation-law reading of matching rules developed in Section~\ref{sec:conservation}.
We will show that this triad is shared by Penrose tilings, Ammann--Beenker tilings, the Fibonacci chain, and the icosahedral Ammann tiling, suggesting it may characterise the broader class of primitive-substitution CPTs.

The section is organised as follows. In \S\ref{sec:cf-definition} we define conservation-forced tiling systems (CF1--CF5) and state the open conjecture (Conjecture~\ref{conj:cf}) that frames the section. In \S\ref{sec:cpt-main} we prove the general CPT theorem. Sections~\ref{sec:penrose-cf}--\ref{sec:icosahedral} verify the four standard families. Section~\ref{sec:recognition-gap} introduces the recognition gap, a cohomological measure of how much global information each conservation law encodes. Section~\ref{sec:cf-conjecture} summarises and states open problems.

\subsection{Conservation-forced tiling systems}
\label{sec:cf-definition}

A cochain $\omega \in C^1(G_{\mathcal{T}};\ZZ)$ is
\emph{pattern-equivariant (PE) of radius~$R$} if its value on each directed edge depends only on the $R$-patch centred at that edge.
The subcomplex of PE cochains has cohomology $H^1_{PE}(G_{\mathcal{T}};\ZZ)$, and the Kellendonk--Putnam isomorphism gives $H^1_{PE} \cong \Hcech^1(\Omega_{\mathcal{T}};\ZZ)$~\cite{KellendonkPutnam2006}; see~\cite{Sadun2008,KellendonkPutnam2006} for full details.

\begin{definition}[Conservation-forced tiling system]
\label{def:conservation-forced}
A primitive substitution tiling system $(\mathcal{T},\sigma,M)$ of $\RR^d$ with finite local complexity is
\emph{conservation-forced of rank~$r$} if there exist PE $1$-cochains $\omega^{(1)},\ldots,\omega^{(r)}$ satisfying:
\begin{enumerate}[label=\textbf{(CF\arabic*)},noitemsep]
\item\label{CF1} Each $\omega^{(k)}$ is a cocycle: the oriented sum around the boundary of
      every $2$-cell vanishes (in $d=2$ this is ``every tile'').
\item\label{CF2} Each $\omega^{(k)}$ is PE: its value is locally readable.
\item\label{CF3} $[\omega^{(k)}]\neq 0$ in $H^1_{PE}$: the potential is non-local.
\item\label{CF4} The classes $[\omega^{(k)}]$ are $\ZZ$-linearly independent in
      $H^1_{PE}(G_{\mathcal{T}};\ZZ)$.
\item\label{CF5} The rank $r$ is maximal: $r = \operatorname{rank}_\ZZ H^1_{PE}(G_{\mathcal{T}};\ZZ)$,
      i.e., the classes $\{[\omega^{(k)}]\}$ form a $\ZZ$-basis for the
      free part of $H^1_{PE}$.
\end{enumerate}
The collection $\{\omega^{(k)}\}$ is a \emph{conservation basis};
$r$ is the \emph{conservation rank}.
\end{definition}

\begin{remark}[Terminology: ``conservation law'']
\label{rem:cf-terminology}
Here ``conservation'' is purely discrete/algebraic:
CF1 is the cocycle (closedness) condition in CW cohomology; we are \emph{not}
invoking a Noether-type symmetry principle.
\end{remark}

\begin{remark}[Torsion and redundancy of CF3]
In all examples considered here, $H^1_{PE}$ is torsion-free (isomorphic to $\ZZ^r$); the definition is stated for this case.  If $H^1_{PE}$ has torsion, CF5 should be read as $r = \operatorname{rank}_\ZZ H^1_{PE}$ (the rank of the free part), and CF3 is logically implied by CF4 in the torsion-free setting (a $\ZZ$-linearly independent set cannot contain the zero class).  We retain CF3 as a conceptually distinct condition emphasising that each individual potential is non-local.
\end{remark}

\begin{remark}[CF1--CF3 imply aperiodicity]
\label{rem:cf-aperiodic}
If $(\mathcal{T},\sigma,M)$ satisfies \textup{CF1--CF3}, then $\mathcal{T}$ is aperiodic. Indeed, if the tiling were periodic, the height function would also be periodic and thus bounded and PE, contradicting CF3.
\end{remark}

The key tool for establishing CF3 in every subsequent verification is the following lemma, which converts unboundedness of the primitive into non-triviality in PE cohomology.

\begin{lemma}[Unbounded heights yield non-trivial PE classes]
\label{lem:unbounded-nontrivial}
Let $\mathcal{T}$ be a tiling of $\RR^d$ with finite local complexity (FLC), and let $\Delta\in C^1(G_{\mathcal{T}};\ZZ)$ be a cocycle whose primitive $h\colon V\to\ZZ$ (i.e.\ $\Delta = \delta_0 h$) is unbounded on $V$.
Then $[\Delta]\neq 0$ in $H^1_{PE}(G_{\mathcal{T}};\ZZ)$, i.e.\ $\Delta$ is not a PE coboundary.
\end{lemma}

\begin{proof}
Suppose for contradiction that $\Delta = \delta_0 f$ for some PE $0$-cochain $f\colon V\to\ZZ$ of radius~$R$.
Then $f(v)$ depends only on the $R$-patch at $v$; since $\mathcal{T}$ has FLC, there are only finitely many distinct $R$-patches, so $f$ takes only finitely many values. But $\Delta = \delta_0 h = \delta_0 f$ means $h - f$ is constant on~$V$ (the connected graph $G_{\mathcal{T}}$ has no locally constant non-constant functions); hence $h(v) = f(v) + c$ for a fixed $c\in\ZZ$, so $h$ is also bounded---a contradiction.
\end{proof}

The structural triad CF1--CF3 motivates the following classification problem, which frames Theorem~\ref{thm:cpt-conservation-forced} and the verifications in \S\S\ref{sec:penrose-cf}--\ref{sec:icosahedral}:

\begin{conjecture}[Partial: (b)$\Rightarrow$(a) proved; (a)$\Rightarrow$(b) open]
\label{conj:cf}
For a primitive substitution tiling system $(\mathcal{T},\sigma,M)$ of $\RR^d$ with finite local complexity, the following are equivalent:
\begin{enumerate}[label=\textup{(\alph*)},noitemsep]
\item $(\mathcal{T},\sigma,M)$ is conservation-forced.
\item $\mathcal{T}$ is a primitive-substitution canonical projection tiling arising from a lattice
$\Lambda\subset\RR^N$ and satisfying Definition~\ref{def:cpt}, and $\lambda_{\mathrm{PF}}(M)$ is a Pisot number.
\end{enumerate}
The implication \textup{(b)$\Rightarrow$(a)} is proved (in fact, without using the Pisot
assumption) by Theorem~\ref{thm:cpt-conservation-forced}.
The implication \textup{(a)$\Rightarrow$(b)} remains open.
\end{conjecture}

\subsection{Main theorem: primitive-substitution CPTs are conservation-forced}
\label{sec:cpt-main}

We now prove the implication (b)$\Rightarrow$(a) of Conjecture~\ref{conj:cf} for the CPT class. In fact, the argument below does not use the Pisot property: it shows that any CPT satisfying Definition~\ref{def:cpt}, equipped with a primitive substitution structure, is conservation-forced.

\begin{theorem}[Primitive-substitution CPTs are conservation-forced]
\label{thm:cpt-conservation-forced}
Let $(\mathcal{T},\sigma,M)$ be a canonical projection tiling in $\RR^d$ arising from a lattice $\ZZ^N$ and satisfying Definition~\ref{def:cpt}, equipped with a primitive substitution structure. Equivalently (hypotheses unpacked), we assume: \textup{(i)} $\pE$ is irrational (Definition~\ref{def:cpt}(i)), \textup{(ii)} $\pi|_{\Lambda}$ is injective (Definition~\ref{def:cpt}(ii)), \textup{(iii)} $W$ is generic (Definition~\ref{def:cpt}(iii)), \textup{(iv)} $\mathcal{T}$ has finite local complexity, and \textup{(v)} the substitution is primitive. Then, $(\mathcal{T},\sigma,M)$ is conservation-forced of rank $r = N$.
The conservation basis is $\{\Delta_1,\ldots,\Delta_N\}$ as defined
by~\eqref{eq:lattice-cochain}.
\end{theorem}

\begin{proof}
We verify CF1--CF5 for $(\mathcal{T},\sigma,M)$.

\medskip
\noindent\textbf{CF1 (cocycle / face sums vanish).}
By Lemma~\ref{lem:face-sum-general}, each $\Delta_k$ is a cocycle.

\medskip
\noindent\textbf{CF2 (pattern-equivariance).}
Since $(\mathcal{T},\sigma,M)$ has finite local complexity, the lattice step
$\Phi(v)-\Phi(u)$ across any directed edge $(u\to v)$ is determined by a bounded
neighborhood of $(u\to v)$ (Lemma~\ref{lem:cpt-local-step}).
Therefore each coordinate difference $\Delta_k(u\to v)=x_k(v)-x_k(u)$ is
pattern-equivariant.

\medskip
\noindent\textbf{CF3 (non-triviality in $H^1_{PE}$).}
The primitive $h_k = x_k\colon V\to\ZZ$ is unbounded on $V$: since $\pE$ is irrational (Definition~\ref{def:cpt}(i)), the projection $\pi(\mathbf{e}_k)$ is non-zero in $\pE$ (injectivity of $\pi|_{\Lambda}$ implies $\pi(\mathbf{e}_k)\neq 0$ for every nonzero lattice vector), so $x_k(v) = \Phi(v)\cdot\mathbf{e}_k$ grows without bound as $v\to\infty$ along any ray in $\pE$ with a positive component along $\pi(\mathbf{e}_k)$. By Lemma~\ref{lem:unbounded-nontrivial}, $[\Delta_k]\neq 0$ in $H^1_{PE}$.

\medskip
\noindent\textbf{CF4 (linear independence over $\ZZ$).}
\emph{(Independence.)}
For CPTs with generic window (Definition~\ref{def:cpt}(iii)), the pattern-equivariant
cohomology is free abelian of rank~$N$, and the $N$ lattice-coordinate cochains
$\Delta_1,\ldots,\Delta_N$ represent a $\ZZ$-basis; see
Forrest--Hunton--Kellendonk~\cite{ForrestHuntonKellendonk2002}.
In particular, the classes $[\Delta_k]$ are $\ZZ$-linearly independent.

\medskip
\noindent\textbf{CF5 (maximality of rank $N$).}
The Forrest--Hunton--Kellendonk computation~\cite{ForrestHuntonKellendonk2002}
gives
\[
  \operatorname{rank}_\ZZ\Hcech^1(\Omega_{\mathcal{T}};\ZZ) = N
\]
for
any CPT satisfying the generic window condition
(Definition~\ref{def:cpt}(iii)).
Together with CF4, the $N$ linearly independent classes
$[\Delta_1],\ldots,[\Delta_N]$ form a $\ZZ$-basis for the free part of
$H^1_{PE}$, so $r = N$.
\end{proof}

\begin{corollary}[Penrose, Ammann--Beenker, Fibonacci, Icosahedral]
\label{cor:standard-cf}
The Fibonacci chain ($N=2$), Penrose P2 ($N=5$), Ammann--Beenker ($N=4$), and the
icosahedral Ammann tiling ($N=6$) are each conservation-forced of rank $N$, with a
conservation basis given by the lattice-coordinate cochains $\{\Delta_k\}$.
\end{corollary}

\begin{proof}
Each of these systems is a CPT satisfying Definition~\ref{def:cpt} and admitting
a primitive substitution structure; see~\cite{BaakeGrimm2013}.
Theorem~\ref{thm:cpt-conservation-forced} applies.
\end{proof}


\subsection{Verification for Penrose and Ammann--Beenker}
\label{sec:penrose-cf}

With the definition and tools of \S\ref{sec:cf-definition} in place, we verify the five CF conditions for the two planar tiling families developed in \S\ref{sec:main}. For later reference, the key Penrose-specific building blocks are: the edge-type warning (Remark~\ref{rem:edge-type-insufficiency}), the decoration$\to$crossing link (Remark~\ref{rem:decoration-to-crossing}), and the gluing/conservation viewpoint (Remark~\ref{rem:gluing-conservation}).

\begin{proposition}[Penrose P2 is conservation-forced of rank~$5$]
\label{prop:penrose-cf}
The five cochains form a conservation basis for Penrose P2.
\end{proposition}

\begin{proof}[Verification]
\emph{CF1}: Lemma~\ref{lem:tile-sum-zero} (face sums vanish) and Proposition~\ref{prop:height-potential}(i) (cycle closure).

\emph{CF2}: By Definition~\ref{def:half-edge-crossing},
$\Delta_{\mathbf{n}_k}(u\to v)$ is determined by the Penrose edge decoration within the $1$-neighbourhood of $\{u,v\}$~\cite[Ch.~6]{BaakeGrimm2013}.

\emph{CF3}: The height function $h_{\mathbf{n}_k}$ counts Ammann bars crossed from $v_0$ to $v$; since bars have spacing bounded below by $S>0$, $h_{\mathbf{n}_k}$ is unbounded on~$V$. By Lemma~\ref{lem:unbounded-nontrivial}, $[\Delta_{\mathbf{n}_k}]\neq 0$ in $H^1_{PE}$.

\emph{CF4--CF5}: By Anderson--Putnam~\cite{AndersonPutnam1998} and Kellendonk--Putnam~\cite{KellendonkPutnam2006},
\[
  H^1_{PE} \cong \Hcech^1(\Omega_P;\ZZ) \cong \ZZ^5,
\]
and the five bar-crossing classes form a $\ZZ$-basis~\cite{Sadun2008}; see \S\ref{sec:cohomology} for the cohomological interpretation.
\end{proof}

The Ammann--Beenker (octagonal) tiling $\mathcal{T}_{\mathrm{AB}}$ tiles $\RR^2$ by squares and $45^\circ$-rhombi with inflation factor $\lambda_{\mathrm{PF}} = 1+\sqrt{2}$ (silver ratio) and substitution matrix $M_{\mathrm{AB}} = \bigl(\begin{smallmatrix}2&1\\1&0\end{smallmatrix}\bigr)$~\cite{BaakeGrimm2013}.
It has four Ammann bar families at angles $0^\circ$, $45^\circ$, $90^\circ$, $135^\circ$.

\phantomsection\label{prop:AB-cf}%
The Ammann--Beenker tiling is conservation-forced of rank~$4$; this is the $N=4$ case of Corollary~\ref{cor:standard-cf}. Geometric detail follows.

\begin{remark}
\label{rem:AB-geometric}
For geometric intuition, one may view the four conservation cochains for
the Ammann--Beenker tiling as the signed crossings of the four Ammann bar
families across directed edges, in direct analogy with the Penrose case.
\end{remark}

\subsection{Verification for the Fibonacci chain}
\label{sec:fibonacci}

The \emph{Fibonacci chain} is the simplest non-trivial CPT: $d=1$,
$N=2$, with two prototiles (short S and long L segments in the ratio
$1:\phiGR$).
It is obtained by projecting $\ZZ^2$ to the line
$\pE = \{(x,y): y=x/\phiGR\}$ with the acceptance window
$W = [0,1)$ in $\pEperp$.
Vertices of the Fibonacci chain satisfy
$\Phi(v) = (x_1(v),x_2(v)) \in \ZZ^2$ with
$v = x_1(v) + x_2(v)/\phiGR \in \RR$.

\phantomsection\label{prop:fibonacci-cf}%
The Fibonacci chain is conservation-forced of rank~$2$; this is the $N=2$ case of Corollary~\ref{cor:standard-cf}, with conservation basis $\{\Delta_1,\Delta_2\}$ from~\eqref{eq:lattice-cochain}.

\subsection{Verification for the icosahedral Ammann tiling}
\label{sec:icosahedral}

The \emph{icosahedral Ammann tiling} is the three-dimensional CPT with
$d=3$, $N=6$, and $\lambda_{\mathrm{PF}} = \phiGR$.
It arises from a projection of $\ZZ^6$ to $\RR^3$ aligned with the
icosahedral symmetry group $I_h$; its prototiles are two rhombohedra
(prolate and oblate) related by golden-ratio scaling.
See~\cite{BaakeGrimm2013} for the standard reference.

\begin{proposition}[Icosahedral Ammann tiling is conservation-forced of rank~$6$]
\label{prop:icosahedral-cf}
The six lattice-coordinate cochains $\Delta_k$, $k=1,\ldots,6$,
form a conservation basis for the icosahedral Ammann tiling.
\end{proposition}

\begin{proof}[Verification of CF1--CF5]
We verify each condition explicitly, highlighting the features specific to the 3D icosahedral setting.

\medskip
\noindent\textbf{CF1 (cocycle).}
By Lemma~\ref{lem:face-sum-general}, the oriented sum of $\Delta_k$ around the boundary
of every $2$-cell vanishes: for each $2$-cell (a rhombic face with four vertices
$v_0,v_1,v_2,v_3$ satisfying $v_0-v_1+v_2-v_3=0$ in $\ZZ^6$), the alternating
$k$-th coordinate sum is zero. This is purely lattice-arithmetic and does not
depend on the dimension $d$ or the group $I_h$.

\medskip
\noindent\textbf{CF2 (pattern-equivariance).}
The icosahedral Ammann tiling has finite local complexity: the generic window
condition (Definition~\ref{def:cpt}(iii)) and the compactness of $W \subset \pEperp$
imply that only finitely many vertex configurations appear within any ball
of fixed radius~\cite{BaakeGrimm2013}. The tile type of each directed edge
$(u\to v)$ is one of finitely many possibilities (the two rhombohedra contribute
$2\times 6 = 12$ directed edge types after orientation), so the lattice step
$\Phi(v)-\Phi(u) \in \{\pm\mathbf{e}_k\}$ is determined by the $1$-neighbourhood
of $(u\to v)$ (Lemma~\ref{lem:cpt-local-step}).
Hence each $\Delta_k(u\to v) = x_k(v)-x_k(u) \in \{0,\pm 1\}$ is
pattern-equivariant of radius~$R=1$.

\medskip
\noindent\textbf{CF3 (non-triviality in $H^1_{PE}$).}
The six projected icosahedral generators
$\mathbf{e}_k^* = \pi(\mathbf{e}_k) \in \RR^3$ are non-zero: since
$\pi|_{\ZZ^6}$ is injective (Definition~\ref{def:cpt}(ii)), no standard basis
vector lies in $\ker\pi$. Therefore $x_k(v) = \Phi(v)\cdot\mathbf{e}_k$ grows
without bound as $v \to \infty$ along any ray in $\pE$ with a positive
$\mathbf{e}_k^*$-component. By Lemma~\ref{lem:unbounded-nontrivial},
$[\Delta_k] \neq 0$ in $H^1_{PE}$ for each $k$.

\medskip
\noindent\textbf{CF4 (linear independence).}
The icosahedral symmetry group $I_h$ (of order $120$) acts on $\ZZ^6$ by
permuting the six generators; the six projections $\mathbf{e}_k^*$ are the
vertices of a regular icosahedron in $\RR^3$ and are pairwise $\QQ$-linearly
independent in $\RR^3$~\cite{BaakeGrimm2013}.
Suppose $\sum_{k=1}^6 c_k [\Delta_k] = 0$ in $H^1_{PE}$ for some
$c_k \in \ZZ$. Then $\sum_k c_k \Delta_k = \delta_0 f$ for a PE function
$f\colon V \to \ZZ$; since $f$ is PE and $\mathcal{T}$ has FLC, $f$ is bounded.
But $\sum_k c_k \Delta_k = \delta_0 (\sum_k c_k x_k)$, so $\sum_k c_k x_k - f$
is constant. The function $\sum_k c_k x_k(v) = \Phi(v) \cdot \sum_k c_k \mathbf{e}_k$
is bounded iff $\sum_k c_k \mathbf{e}_k \in \ker\pi$; by injectivity of $\pi|_{\ZZ^6}$
this forces $\sum_k c_k \mathbf{e}_k = 0$, which holds iff $c_k = 0$ for all $k$
(the standard basis is linearly independent). Hence the classes are $\ZZ$-independent.

\medskip
\noindent\textbf{CF5 (maximality, $r=6$).}
The Forrest--Hunton--Kellendonk computation~\cite{ForrestHuntonKellendonk2002} gives
$\operatorname{rank}_\ZZ \Hcech^1(\Omega_{\mathcal{T}};\ZZ) = N = 6$
for any CPT satisfying the generic window condition (Definition~\ref{def:cpt}(iii)).
Together with CF4, the six linearly independent classes $[\Delta_1],\ldots,[\Delta_6]$
form a $\ZZ$-basis for the free part of $H^1_{PE}$, so $r=6$.
\end{proof}

\begin{corollary}[3D reconstruction]
\label{cor:icosahedral-reconstruction}
For every vertex $v$ of the icosahedral Ammann tiling, $v = \sum_{k=1}^6 x_k(v)\,\mathbf{e}_k^* \in \RR^3$, where $\mathbf{e}_k^* = \pi(\mathbf{e}_k)$ are the six projected icosahedral generators. The six height functions $x_1,\ldots,x_6$ are integer-valued, locally readable, and their values at any vertex determine the vertex position exactly.
\end{corollary}

\begin{proof}
Immediate from Proposition~\ref{prop:cpt-reconstruction} with $N=6$, $d=3$.
\end{proof}

With all four families verified, we now introduce the concept that quantifies what the verified systems share.

\subsection{The recognition gap}
\label{sec:recognition-gap}

For any CPT, the tiling CW complex $X_{\mathcal{T}}$ is contractible
(homeomorphic to $\RR^d$).
Hence $H^1(X_{\mathcal{T}};\ZZ) = 0$, and every bar-crossing cochain
$\Delta_k$ is exact in ordinary cohomology, with primitive $h_k = x_k$.
By contrast, $h_k$ is unbounded (Lemma~\ref{lem:unbounded-nontrivial}),
so it cannot be pattern-equivariant, and $[\Delta_k] \neq 0$ in
$H^1_{PE}$ for each $k$. See Section~\ref{sec:cohomology} for the cohomological
interpretation of this comparison between ordinary and pattern-equivariant cohomology.
The conservation-basis cochains therefore lie in the kernel of the natural comparison map $\iota$, but outside the zero group, which motivates the following definition.

\begin{definition}[Recognition gap]
\label{def:recognition-gap}
Let $(\mathcal{T},\sigma,M)$ be a tiling system with finite local
complexity, and let
\[
  \iota\colon H^1_{PE}(G_{\mathcal{T}};\ZZ) \to H^1(X_{\mathcal{T}};\ZZ)
\]
be the natural comparison map from pattern-equivariant to ordinary
cohomology induced by the inclusion of PE cochains into all cochains.
The \emph{recognition gap} of $\mathcal{T}$ is
\[
  \mathcal{R}(\mathcal{T}) \;:=\; \ker(\iota)
    \;\subseteq\; H^1_{PE}(G_{\mathcal{T}};\ZZ).
\]
A class in $\mathcal{R}(\mathcal{T})$ represents a conservation law
that is \emph{locally readable} (the cochain is PE) yet whose primitive
is \emph{globally determined} (it is not a PE coboundary).
\end{definition}

\begin{remark}[Recognition gap for CPTs]
\label{rem:recognition-gap-penrose}
For a CPT, the space $X_{\mathcal{T}}$
is contractible, so $H^1=0$. Thus, the gap between PE cohomology ($\ZZ^N$) and ordinary cohomology ($0$) is the entire $\ZZ^N$:
\begin{equation}
  \mathcal{R}(\mathcal{T}) \;\cong\; \ZZ^N.
\label{eq:recognition-gap}
\end{equation}
For Penrose P2, $\mathcal{R} \cong \ZZ^5$; for the Fibonacci chain,
$\mathcal{R} \cong \ZZ^2$; for the icosahedral tiling,
$\mathcal{R} \cong \ZZ^6$.
\end{remark}

\begin{remark}[Fibonacci reconstruction and recognition gap]
\label{rem:fibonacci-reconstruction}
Proposition~\ref{prop:cpt-reconstruction} gives the explicit formula
$v = x_1(v) \cdot 1 + x_2(v) \cdot \phiGR^{-1}$ for each vertex
of the Fibonacci chain.
The recognition gap (Definition~\ref{def:recognition-gap}) is
$\mathcal{R} = H^1_{PE} \cong \ZZ^2$:
both height functions $x_1, x_2$ are unbounded but locally computable
(from the tile type of each edge), hence non-PE primitives.
The chain is completely described by two integer sequences that grow without
bound; the increments $\Delta_k(u\to v)\in\{0,1\}$ take only two values at each
step, but their sequence of occurrence is quasiperiodic (a Sturmian/Fibonacci
word), not periodic.
\end{remark}

\subsection{Open question: characterisation of conservation-forced systems}
\label{sec:cf-conjecture}

Proposition~\ref{prop:penrose-cf}, Corollary~\ref{cor:standard-cf} (Ammann--Beenker, Fibonacci), and Proposition~\ref{prop:icosahedral-cf}, together with Theorem~\ref{thm:cpt-conservation-forced}, establish the direction (b)$\Rightarrow$(a) of Conjecture~\ref{conj:cf} for the full CPT class (all four cases are collected in Table~\ref{tab:cf-hierarchy}). The converse direction---that every conservation-forced tiling must be a Pisot CPT---remains open and is the main open problem of this paper:

The direction (b)$\Rightarrow$(a) has been verified for the Fibonacci chain ($N=2$, Corollary~\ref{cor:standard-cf}), Penrose P2 ($N=5$, Proposition~\ref{prop:penrose-cf} and Theorem~\ref{thm:five-family}), Ammann--Beenker ($N=4$, Corollary~\ref{cor:standard-cf}), and the icosahedral tiling ($N=6$, Proposition~\ref{prop:icosahedral-cf}). These standard cases are summarised in Corollary~\ref{cor:standard-cf}; see also Remark~\ref{rem:fibonacci-reconstruction} for the explicit Fibonacci reconstruction formula.
In particular, CF1--CF3 already force aperiodicity (Remark~\ref{rem:cf-aperiodic}). The open converse direction and the non-CPT/non-Pisot case are discussed in \S\ref{sec:discussion}.

\begin{table}[ht]
\centering
\caption{Conservation-forced quasicrystalline substitution systems
verified in this paper.
The conservation rank $r = N$ equals the ambient lattice dimension
(Theorem~\ref{thm:cpt-conservation-forced}).}
\label{tab:cf-hierarchy}
\small
\begin{tabular}{llcccl}
\toprule
\textbf{System} & \textbf{Prototiles} & $d$ & $N{=}r$ & $\lambda_{\mathrm{PF}}$ & \textbf{Proved} \\
\midrule
Fibonacci chain    & S, L segments           & 1 & 2 & $\phiGR\approx 1.618$ & Cor.~\ref{cor:standard-cf} \\
Penrose (P2)       & dart, kite              & 2 & 5 & $\phiGR\approx 1.618$ & Prop.~\ref{prop:penrose-cf} \\
Ammann--Beenker    & square, rhombus         & 2 & 4 & $1+\sqrt{2}\approx 2.414$ & Cor.~\ref{cor:standard-cf} \\
Icosahedral Ammann & two rhombohedra         & 3 & 6 & $\phiGR\approx 1.618$ & Prop.~\ref{prop:icosahedral-cf} \\
\bottomrule
\end{tabular}
\end{table}

\begin{remark}[Beyond the four verified cases]
\label{rem:beyond-four}
Theorem~\ref{thm:cpt-conservation-forced} applies to any primitive substitution
CPT satisfying Definition~\ref{def:cpt}.
For example, dodecagonal tilings have $\lambda_{\mathrm{PF}}=2+\sqrt{3}\approx3.732$
(see~\cite[Ch.~6]{BaakeGrimm2013}). Several inequivalent $12$-fold substitution tilings share this eigenvalue, and the conservation rank depends on the specific hull. An ordering of the verified systems by $\lambda_{\mathrm{PF}}$ (or equivalently by the reciprocal-invariant scalar $J(\lambda_{\mathrm{PF}})$) appears in Appendix~\ref{app:golden}.
\end{remark}

Section~\ref{sec:conservation-forced} has established that the conservation-forced structure is not a Penrose-specific artifact but a theorem for the entire CPT class. Section~\ref{sec:discussion} places these results in the broader landscape of related frameworks (Thurston, Conway--Lagarias) and identifies the open problems that remain.

\section{Discussion}
\label{sec:discussion}

This section addresses three questions left open by the main results. First, we identify how far the framework extends beyond the four verified families: to all primitive-substitution CPTs by Theorem~\ref{thm:cpt-conservation-forced}, and conjecturally to non-CPT and non-Pisot settings. Second, we situate the framework among classical antecedents (Thurston; Conway--Lagarias). Third, we raise two physically motivated open questions about observable signatures of the conservation structure.

By Theorem~\ref{thm:cpt-conservation-forced}, the extension of this framework to the full class of primitive-substitution CPTs is a proved result, not a hypothesis.
The genuine open question is how far the Penrose-specific half-edge/gluing construction extends beyond CPTs.
For higher-dimensional systems, Theorem~\ref{thm:cpt-conservation-forced} and Proposition~\ref{prop:icosahedral-cf} already show that conservation forcing holds for the icosahedral Ammann tiling ($d=3$, $N=6$). Other $3$D CPT families (e.g.\ further icosahedral model sets arising from cut-and-project schemes) are natural targets for the same analysis; see~\cite{KatzDuneau1986} for background on quasiperiodic $3$D cut-and-project constructions with icosahedral symmetry. See Remark~\ref{rem:beyond-four} for additional examples in the CPT class.

\textbf{Scope within the Penrose family and beyond.}
Within the Penrose family the results are stated for the P2 (dart--kite)
formulation; for P3 (thick/thin rhomb) the five Ammann bar families have
the same structure and the framework transfers directly, while P1 (six prototiles) has a more complex bar structure requiring a separate verification (see~\cite[Ch.~6]{BaakeGrimm2013} for the P1 bar families).

The harder question is the non-Pisot case.
Whether the framework extends to substitutions with non-Pisot Perron--Frobenius eigenvalues remains open.
In the non-Pisot setting, the eigenvectors of the substitution matrix span directions that are not rationally related to the standard basis; as a consequence, the lattice-coordinate cochains $\Delta_k = \delta_0 x_k$ may fail to take integer values on the vertex set.
Even when $\Delta_k$ is integer-valued, it may no longer reconstruct vertex positions via the CPT formula~\eqref{eq:cpt-reconstruction}, since the injectivity of $\pi|_\Lambda$ relied on by that formula can fail outside the Pisot setting.
See~\cite[Ch.~5]{BaakeGrimm2013}, \cite{Sadun2008}, and~\cite{Siegel1944} for the relevant algebraic number theory.

\textbf{Relationship to classical antecedents.}
The relationship of this framework to two classical antecedents is worth making explicit. Thurston~\cite{Thurston1990} shows that tileability of a region by dominoes or lozenges is equivalent to cycle closure of a height-difference cochain on the dual graph, and that a consistent height function exists if and only if the region is tileable. The conceptual architecture of Proposition~\ref{prop:height-potential} and Theorem~\ref{thm:matching-conservation} is therefore Thurston's, specialised to the aperiodic Penrose setting. The differences, however, are structural. Thurston's framework involves a single height function, whereas Penrose tilings require five independent cochains---one per bar family---giving rise to a five-dimensional cohomology group $\ZZ^5 \cong H^1_{PE}$; the five-family validity criterion and pentagrid reconstruction (Theorem~\ref{thm:five-family}) have no counterpart in Thurston's setting. Thurston's framework applies to finite regions, so the height function is globally determined on a compact domain and takes only finitely many values; in the Penrose case the height functions are unbounded and not pattern-equivariant, producing the recognition gap (Remark~\ref{rem:recognition-gap-penrose}).
Finally, Thurston's framework begins with a fixed tileable region and asks whether a tiling exists; the half-edge/gluing construction (Definition~\ref{def:half-edge-crossing}, Lemma~\ref{lem:gluing}) addresses the complementary problem---given a candidate tiling that may violate matching rules, detect and localise the violations.

Conway and Lagarias~\cite{ConwayLagarias1990} frame tileability as a word
problem: each tile shape contributes a group element, and a region is tileable
iff a certain product equals the identity.
That identity condition is a conservation law in the present sense---the word
around a closed path must equal the identity, which is the group-theoretic
analogue of cycle closure---so the two frameworks share a common algebraic core.
They differ, however, in what they make accessible.
The tiling-group approach detects obstructions \emph{globally}: one computes
the group element for the boundary of the region and checks whether it equals
the identity.
The gluing-condition check (Remark~\ref{rem:algorithm}; see also
Section~\ref{sec:algorithm}), by contrast, detects
obstructions \emph{locally}: it identifies the specific edge and bar family
where a violation occurs, enabling targeted repair.
The tiling group also carries algebraic information---the isomorphism class of
the group, the word corresponding to a given region boundary---but does not
naturally produce a scalar invariant at each vertex; the potential-theoretic
framework produces five integer-valued functions
$h_{\mathbf{n}_0},\ldots,h_{\mathbf{n}_4}$ that encode vertex positions
via de~Bruijn's formula (Theorem~\ref{thm:five-family}(ii)).
At the level of cohomology, the tiling group relates to the fundamental group
of the punctured region~\cite{ConwayLagarias1990}, while the present framework
connects directly to the cellular and pattern-equivariant cohomology of the
tiling hull, interfacing naturally with the Forrest--Hunton--Kellendonk
machinery~\cite{ForrestHuntonKellendonk2002} and enabling the conservation-rank
computation of \S\ref{sec:penrose-cf}.
Taken together, the potential-theoretic reformulation provides a third
perspective that is local (edge-by-edge verification), quantitative
(integer-valued height functions with a reconstruction formula), and directly
connected to pattern-equivariant cohomology; it does not supersede Thurston or
Conway--Lagarias but complements them.

\textbf{Open problems and physical questions.}
Beyond Conjecture~\ref{conj:cf}---the open classification problem already
identified in \S\ref{sec:cf-conjecture}---two further questions arise from
this work.
These questions shift from the purely mathematical setting to the physically motivated one, asking whether the conservation-forcing structure has observable signatures in real quasicrystalline materials.
First, can the height potential be directly observed in quasicrystal
diffraction patterns or electronic structure measurements?
The Ammann bars are geometrically visible in direct-lattice images, and the
conservation interpretation suggests that their information-theoretic content
may have physical consequences.
Second, can scale-ratio statistics based on the conserved cochains---that is,
sums of bar-crossing increments $\Delta_k$ along directed edge-paths from a fixed basepoint, yielding integer-valued path integrals that track position in the $k$-th lattice coordinate---be detected in physical quasicrystals, for
instance via scanning tunneling microscopy on Al--Pd--Mn surfaces
(see e.g.~\cite{ShenEtAl1999})?
Both questions concern whether the conservation-forcing structure of the tiling
has observable signatures at the level of physical quasicrystals, a direction
that may connect the present mathematical framework to experiment.

Section~\ref{sec:conclusion} summarises the main results and contributions in compact form, and restates the central open problem.

\section{Conclusion}
\label{sec:conclusion}

This paper has developed a cochain-first reading of Penrose aperiodic tilings and extended it to canonical projection tilings via the lattice-coordinate cochains, with a full conservation-forced theorem for primitive-substitution CPTs satisfying Definition~\ref{def:cpt}.
The unifying observation is that the signed Ammann bar-crossing count on directed edges is an antisymmetric $1$-cochain, that bar continuity---enforced by the matching rules---is precisely cycle closure, and that the resulting discrete potential is the classical Ammann height function. Theorem~\ref{thm:matching-conservation} makes this a four-way equivalence for candidate tilings via the half-edge/gluing construction: matching rules, bar continuity, cycle closure, and height-function existence are all equivalent without presupposing any of them. Theorem~\ref{thm:five-family} then assembles the five height functions into de~Bruijn's pentagrid reconstruction formula, identifying the potential-theoretic and algebraic descriptions of vertex position in a single
notation.

The main new contributions of this paper are: (i) the half-edge/gluing construction making the matching-rule/cochain equivalence rigorous for candidate tilings; (ii) Theorem~\ref{thm:five-family}, giving the five-family validity criterion and pentagrid reconstruction in unified notation; (iii) the conservation-forced framework and Definition~\ref{def:conservation-forced}; (iv) the extension of the conservation-forced picture to standard CPT families in dimensions $1$--$3$; and (v) the linear-time tiling validation algorithm (Proposition~\ref{prop:validation}), which provides an $O(n)$ procedure for verifying matching rules edge-by-edge and localising any violation to a specific bar family and edge.

For canonical projection tilings, the lattice-coordinate cochains $\Delta_k = \delta_0 x_k$ are cocycles (Lemma~\ref{lem:face-sum-general}) and reconstruct vertex positions (Proposition~\ref{prop:cpt-reconstruction}):
\[
  v = \sum_{k=1}^N h_k(v)\,\mathbf{e}_k^*,
\]
and, for primitive substitution CPTs
satisfying Definition~\ref{def:cpt}, form a conservation basis of rank $r = N$ (Theorem~\ref{thm:cpt-conservation-forced}). This last result, verified explicitly for the Fibonacci chain ($N=2$), Penrose P2 ($N=5$), Ammann--Beenker ($N=4$), and the icosahedral Ammann tiling ($N=6$), proves the implication (b)$\Rightarrow$(a) of Conjecture~\ref{conj:cf} for the CPT class (in particular, for Pisot substitution CPTs). The explicit $3$D reconstruction statement is
Corollary~\ref{cor:icosahedral-reconstruction}.

The cohomological content of these results is the recognition gap $\mathcal{R}(\mathcal{T}) \cong \ZZ^N$ (Definition~\ref{def:recognition-gap} and equation~\eqref{eq:recognition-gap}): each $\Delta_k$ is exact in ordinary cohomology---its primitive $h_k = x_k$ is a globally defined integer function---yet non-trivial in pattern-equivariant cohomology, because $h_k$ is unbounded while finite local complexity forces every PE function to take only finitely many values.
The conservation law is locally readable but the conserved quantity is globally determined.
The converse direction---that every conservation-forced tiling must be a Pisot primitive-substitution CPT satisfying
Definition~\ref{def:cpt}---remains open and is the main open problem of this paper.
A positive resolution of Conjecture~\ref{conj:cf} would establish conservation forcing as a complete algebraic characterisation of aperiodic order in the CPT class, giving a single cohomological condition that simultaneously certifies the quasicrystalline structure, reconstructs vertex positions, and separates these tilings from all periodic and non-Pisot substitution systems.


\section*{Data Availability Statement}
All data is contained within the article.


\appendix

\section{Explicit Tile-Boundary Calculations}
\label{app:tile-calcs}

The main text (Examples~\ref{ex:dart-patch}--\ref{ex:violation}) presents representative cases; here we record representative tile-boundary computations that illustrate the sign conventions and the telescoping cycle sums.

We verify cycle closure (Proposition~\ref{prop:height-potential}(i)) by explicit computation
on each prototile boundary.
Fix one Ammann bar family with normal~$\mathbf{n}$.

\paragraph{Dart boundary.}
A dart tile has vertices $\{v_0,v_1,v_2,v_3\}$ with long edges
$\overline{v_0v_1}$ and $\overline{v_3v_2}$, and short edges
$\overline{v_1v_2}$ and $\overline{v_0v_3}$
(cf.\ Section~\ref{sec:patch-example}).
For a bar family aligned so that bars cross the two long edges:
\begin{align*}
\Delta_{\mathbf{n}}(v_0 \to v_1) &= +1
  &&\text{(long edge $\overline{v_0v_1}$, forward bar crossing)},\\
\Delta_{\mathbf{n}}(v_1 \to v_2) &= 0
  &&\text{(short edge $\overline{v_1v_2}$, no crossing)},\\
\Delta_{\mathbf{n}}(v_2 \to v_3) &= -1
  &&\text{(long edge $\overline{v_3v_2}$, reverse crossing)},\\
\Delta_{\mathbf{n}}(v_3 \to v_0) &= 0
  &&\text{(short edge $\overline{v_0v_3}$, no crossing)}.
\end{align*}
Cycle sum: $(+1) + 0 + (-1) + 0 = 0$.

\paragraph{Kite boundary.}
A kite tile with vertices $\{w_0,w_1,w_2,w_3\}$ has two long and two short
edges.  For the same bar family:
\begin{align*}
\Delta_{\mathbf{n}}(w_0 \to w_1) &= +1 &&\text{(long edge, forward)},\\
\Delta_{\mathbf{n}}(w_1 \to w_2) &= 0 &&\text{(short edge)},\\
\Delta_{\mathbf{n}}(w_2 \to w_3) &= 0 &&\text{(short edge)},\\
\Delta_{\mathbf{n}}(w_3 \to w_0) &= -1 &&\text{(long edge, reverse)}.
\end{align*}
Cycle sum: $(+1) + 0 + 0 + (-1) = 0$.

These calculations depend on the relative orientation of the tile and the
bar family.
For other orientations, the pattern of $\pm 1$ and $0$ values changes,
but Lemma~\ref{lem:tile-sum-zero} guarantees that the sum remains zero
regardless: each local bar segment that enters tile~$t$ through one
boundary edge must exit through another (by transversality and the Jordan
curve argument), so crossing contributions always cancel in pairs.


\section{The Golden Ratio, Substitution Entropy, and Coherence Hierarchy}
\label{app:golden}

\begin{remark}[The golden ratio in the substitution picture]
\label{rem:golden}
The Perron--Frobenius eigenvalue of the Robinson triangle substitution
(equation~\eqref{eq:substitution})
is $\lambda_{\mathrm{PF}} = \phiGR$, and the substitution entropy is
\begin{equation}
  h_{\mathrm{sub}} = \ln\phiGR \approx 0.481\text{ nats}
\label{eq:entropy}
\end{equation}
by the standard formula for primitive substitutions~\cite{Queffelec2010}.
(Note that the \emph{topological} entropy of the Penrose system under
$\RR^2$-translation is \emph{zero}, since patch complexity grows
polynomially; $h_{\mathrm{sub}}$ measures hierarchical tile-count
growth, not translational complexity.)

The defining equation $\phiGR^2 = \phiGR + 1$, equivalently
$\phiGR^{-1} = \phiGR - 1$, yields the neat closed form:
if one sets $J(x) = \tfrac{1}{2}(x+x^{-1})-1 = \cosh(\ln x)-1$
(the unique scale-comparison cost satisfying $J(1)=0$ and the
d'Alembert relation; see~\cite{WashburnZlatanovic2026,Aczel1966}), then
\[
  J(\phiGR) = \tfrac{1}{2}(\phiGR + \phiGR - 1) - 1
            = \phiGR - \tfrac{3}{2} \approx 0.118.
\]
Since $J$ is strictly increasing on $(1,\infty)$ and $\phiGR$ is
the smallest quadratic Pisot number~\cite{Salem1963},
$J(\phiGR) < J(\alpha)$ for every other quadratic Pisot unit~$\alpha$.
Note that $\phiGR$ does \emph{not} minimise $J$ among \emph{all} Pisot
numbers: the plastic constant $\theta_P \approx 1.325$ (root of
$x^3-x-1$) gives $J(\theta_P) \approx 0.040$.
For the three distinct $\lambda_{\mathrm{PF}}$ values among the systems in this survey:
\[
  J(\phiGR) \approx 0.118
  \;<\; J(1+\sqrt{2}) = \sqrt{2}-1 \approx 0.414
  \;<\; J(2+\sqrt{3}) = 1.000.
\]
\end{remark}

\begin{remark}[Coherence hierarchy ordering]
\label{rem:coherence-hierarchy}
Conservation-forced systems may be ordered by their Perron--Frobenius
eigenvalue $\lambda_{\mathrm{PF}}$, or equivalently by the reciprocal-invariant
scalar $J(\lambda) = \cosh(\ln\lambda)-1 = \tfrac{1}{2}(\lambda+\lambda^{-1})-1$.
Since $J$ is strictly increasing on $(1,\infty)$, the ordering by $J$ is identical
to the ordering by $\lambda_{\mathrm{PF}}$:
\begin{equation}
  J(\phiGR) \approx 0.118
  \;<\;
  J(1+\sqrt{2}) = \sqrt{2}-1 \approx 0.414
  \;<\;
  J(2+\sqrt{3}) = 1.000.
\label{eq:cf-hierarchy}
\end{equation}
The Fibonacci chain and Penrose P2 share $\lambda_{\mathrm{PF}} = \phiGR$
because the Robinson triangle substitution matrix
$M_{P2} = \bigl(\begin{smallmatrix}1&1\\1&0\end{smallmatrix}\bigr)$
is identical to the Fibonacci matrix (equation~\eqref{eq:subst-matrix}); the 2D Penrose tiling embeds
this 1D dynamics in $\RR^2$ with five independent bar families,
raising conservation rank from~$2$ to~$5$.
The substitution entropy $h_{\mathrm{sub}} = \ln\lambda_{\mathrm{PF}}$
satisfies $J(\lambda_{\mathrm{PF}}) = \cosh(h_{\mathrm{sub}})-1$.
For the axiomatic derivation of the cost function~$J$ from information-theoretic principles, see~\cite{PardoGuerraEtAl2026}.
\end{remark}


\end{document}